\documentclass[a4paper,12pt,twoside]{article}
\usepackage{wiaspreprint}
\usepackage[utf8]{inputenc}
\usepackage[T1]{fontenc}
\usepackage[english]{babel}
\usepackage{amsmath,amssymb,amsfonts,amsthm,mathptmx}

\theoremstyle{remark}
\newtheorem{remark}{Remark}[section]
\theoremstyle{definition}
\newtheorem{theorem}{Theorem}[section]
\newtheorem{definition}[theorem]{Definition}
\newtheorem{example}[theorem]{Example}
\newtheorem{proposition}[theorem]{Proposition}
\newtheorem{lemma}[theorem]{Lemma}
\newtheorem{corollary}[theorem]{Corollary}

\allowdisplaybreaks
\DeclareMathOperator{\R}{\mathbb{R}}

\DeclareMathOperator{\C}{\mathcal{C}}

\DeclareMathOperator{\N}{\mathbb{N}}
\DeclareMathOperator{\BV}{{BV}([0,T])}

\DeclareMathOperator{\ra}{\rightarrow}
\DeclareMathOperator{\de}{\text{d}}

\newcommand{\sym}{\text{sym}}

\newcommand{\D}{\mathcal{D}(\mathcal{K})}
\newcommand{\Do}{\mathcal{D}(\mathcal{K}_0)}

\DeclareMathOperator{\loc}{loc}

\newcommand{\f}[1]{{\pmb{ #1}}}
\DeclareMathOperator{\di}{\nabla\cdot}

\DeclareMathOperator{\clos}{clos}
\DeclareMathOperator{\Ls}{Ls}
\DeclareMathOperator{\Li}{Li}

\newcommand{\tu}{\tilde{\f u}}

\newcommand{\ov}[1]{\overline{{#1}}}

\renewcommand{\t}{\partial_t}

\newcommand{\vv}{\tilde{\f v}}

\newcommand{\V}{H^1_{0,\sigma}(\Omega)}

\newcommand{\Ha}{{L}^2_{\sigma}(\Omega)}

\newcommand{\e}{\text{\textit{e}}}

\title[Energy-variational solutions]{On the existence of weak solutions in the context of multidimensional incompressible fluid dynamics}

\begin{document}
\pagefootright{Berlin, April 22, 2021/rev. September 3, 2021}

\author[R. Lasarzik]{
Robert Lasarzik\nofnmark\footnote{Weierstrass Institute \\
Mohrenstr. 39 \\ 10117 Berlin \\ Germany \\
E-Mail: robert.lasarzik@wias-berlin.de}}
\nopreprint{2834}	
\nopreyear{2021}	
\selectlanguage{english}		
\date{April 22, 2021 (revision: September 3, 2021)}			
\maketitle
\begin{abstract}
We define the concept of energy-variational solutions for the Navier--Stokes and Euler equations. 
The underlying relative energy inequality holds as an  equality for classical solutions and if the additional variable vanishes, these solutions are equivalent to the weak formulation with the strong energy inequality. 
By introducing an additional defect variable in time, all restrictions and all concatenations of energy-variational solutions are again energy-variational solutions. 
Via the criterion of maximal dissipation, a unique solution is selected that is not only continuously depending on the data but also turns out to be a unique weak solution. 

\end{abstract}

\tableofcontents

\section{Introduction}
The Navier--Stokes and Euler equations are the standard models for incompressible fluid dynamics. 
Both are recurrent tools in computational fluid dynamics for weather forecast, micro fluidic devices~\cite{applfluid} or industrial processes like steel production~\cite{najib}. 
There exists a vast literature concerning the Navier--Stokes and Euler equations.
In case of the Navier--Stokes equation,  we only mention here the existence proof for \textit{weak solutions} in three dimension by Leray~\cite{leray} and the \textit{weak-strong uniqueness} result due to Serrin~\cite{serrin}. 
In the context of the Euler equations, the existence of weak solutions in any space dimension is already known for special initial data (see~\cite{convint}) also fulfilling the energy inequality~(see~\cite{convint2}).
This result was proven via the convex integration technique.
This technique grants the existence of infinitely many and also non-physical weak solutions. Additionally, it was proven for the Navier--Stokes equations via similar techniques that there exist infinitely many weak solutions that do not fulfill the energy inequality~\cite{buckmaster}. 
But what is lacking in the literature so far is an existence result for the Navier--Stokes equations in space dimensions larger than four and for the Euler equation with general initial data. Revisiting the previously introduced dissipative solutions for the equations of incompressible fluid dynamics, we refine this concept by introducing \textit{energy-variational solutions}.
%
%
%
%
 As the name already suggests, this notion of generalized solutions is based on a \textit{variation of the underlying energy-dissipation principle}. 
The relative energy inequality can be seen as a variation of the energy-dissipation principle with respect to sufficiently regular functions. 


Dissipative solutions were proposed by  P.-L.~Lions~\cite[Sec.~4.4]{lionsfluid}  in the context of the Euler equations. 
The current author applied this concept in the context of nematic liquid crystals~\cite{diss} and nematic electrolytes~\cite{nematicelectro}. 
It was observed that natural discretizations complying with the properties of the system, like energetic or entropic principles, as well as algebraic restrictions converge naturally to a dissipative solution instead of a measure valued solution (see~\cite{nematicelectro} and~\cite{approx} for details).
In comparison to measure-valued solutions, the degrees of freedom are heavily reduced and no defect measures occur, which are especially difficult to approximate. 
The relative energy inequality, which is at the heart of the dissipative and energy-variational solution concept is also a recurrent tool in PDE theory to prove for instance weak-strong uniqueness~\cite{weakstrong}, stability of stationary states~\cite{diss}, convergence to singular limits~\cite{singular}, or to design optimal control schemes~\cite{approx}. 
An advantage in comparison to distributional or measure-valued solutions is that the solution set inherits the convexity of the energy and dissipation functional, which permits to define appropriate uniqueness criteria~\cite{maxdiss}. 

The definition of energy-variational solutions follows a similar idea as the definition of dissipative solutions, both rely on the so-called relative energy inequality, which compares the solution to smooth test functions fulfilling the PDE only approximately.
But the relative energy inequality for energy-variational solutions is refined such that the resulting inequality becomes an equality for smooth solutions. The nonlinear-convective terms are not only estimated by the relative energy but included in the underlying dissipation potential.
Furthermore, the relative energy inequality holds for all given intervals $(s,t)\subset [0,T]$. This is achieved by introducing a defect variable in time, which measures the difference of weak and strong convergence in the energy in every point in time. Therewith every concatenation and restriction of the solutions to a sub or super time interval is a solution again. This gives rise to the so-called semi-flow property. 
Still the properties of the relative energy inequality remain present, it is preserved for sequences converging in the weak topologies of the associated natural energy and dissipation spaces. Thus in comparison to standard weak solutions,  energy-variational solutions have the advantage that no strong convergence is needed in order to pass to the limit in this formulation. Only Helly's selection principle is used in order to infer the existence of the additional defect variable. 
The existence result only relies on  standard constructive proofs, \textit{i.e.}, a Galerkin discretization in the case of the Navier--Stokes equations and the vanishing viscosity limit in the case of the Euler equations.

Since the energy and dissipation functionals in the considered cases are convex, the set of energy-variational solutions is convex and weakly$^*$ closed. This allows to identify selection criteria in order to select the physically relevant solution. Following the ideas of~\cite{maxbreit,maximaldiss,dafermosscalar,maxdiss}, we propose the selection principle of maximal dissipation. 
This says that the physically relevant solution dissipates energy at the highest rate, hence minimizes the energy in every point in time. This principle becomes even more apparent in thermodynamical consistent systems, where the maximized dissipation implies  maximal entropy (see for instance~\cite{generic} and~\cite[Sec.~9.7]{dafermosbook}).

In~\cite{maxdiss}, the set of dissipative solutions together with the time integral of the energy functional is identified as a suitable convex structure on which such a minimization problem can be defined. The resulting maximally dissipative solution is indeed well-posed in the sense of Hadamard.  
The result of the article at hand applies this technique to energy-variational solutions and the selected unique solution inherits the semi-flow property. 
The selected solution is thus well-posed and unique. In comparison to~\cite{maxdiss}, the energy-variational solution concept allows to define the selection criterion locally in time and the selected minimal energy-variational solution can be shown to be a weak solution.

It is worth noticing that in the framework of minimal energy-variational solutions it is possible to pass to the limit in the quadratic convection term without any strong compactness argument. Only arguments from the direct method of the calculus of variations are needed. 
It is possible to pass to the limit in the quadratic term, first using an additional variable, which catches the difference between weak and strong convergence of the energy. But minimizing the energy afterwards, implies that the additional variable is actually zero. This can be interpreted as additional regularity that the minimizer inherits or that a minimizing sequence of the energy functional actually converges strongly due to the uniform convexity of the underlying energy norm. 
This provides a new tool for the existence of minimal energy-variational and thus weak solutions to nonlinear evolution equations. 
Usually compact embeddings and \textit{a priori} estimates of the time derivative are used to infer strong convergence via some Aubin--Lions argument (compare to~\cite{temam}). 
These ingredients are irrelevant in the present proof, since it only relies on weak convergence in natural spaces and the weakly-lower semi-continuity of the underlying energy and dissipation functionals. 
The proposed technique seems to be very powerful and easily adapted to other systems of PDEs. 
Hence, this gives hope that the new approach 
may allows to prove the existence of  minimal energy-variational and thus weak solutions to some PDE systems, where this seems to be out of reach with other available techniques. This includes multidimensional conservation laws~\cite{maxbreit}, liquid crystals~\cite{weakstrong},  heat-conducting complex fluids~\cite{nsch}, or GENERIC systems in general (see~\cite{generic} and~\cite{maxdiss}).



\textbf{Plan of the paper: }
After providing some notation and preliminaries in Section~\ref{sec:not},  the different solution concepts of weak, energy-variational and minimal energy-variational solutions are defined in Section~\ref{sec:def}. 
Then, we state the main Theorems in Section~\ref{sec:main} and prove them afterwards (see Section~\ref{sec:proof}).

\section{Definitions and main theorems\label{sec:nav}}
\subsection{Preliminaries\label{sec:not}}
Before, we provide the definitions and main results, we collect some notation and preliminary results. 

 \textbf{Notations:}
Throughout this paper, let $\Omega \subset \R^d$ be a {bounded}  Lipschitz domain with $d \geq  2$.
The space of smooth solenoi\-dal functions with compact support is denoted by $\mathcal{C}_{c,\sigma}^\infty(\Omega;\R^d)$. By $L^2_{\sigma}( \Omega) $ and  $\V$ we denote the closure of $\mathcal{C}_{c,\sigma}^\infty(\Omega;\R^d)$ with respect to the norm of $ L^2(\Omega) $ and $  H^1( \Omega) $, respectively.
Note that $L^2_{\sigma}(\Omega) $ can be characterized by $ L^2_{\sigma}(\Omega) = \{ \f v \in L^2(\Omega)| \di \f v =0 \text{ in }\Omega\, , \f n \cdot \f v = 0 \text{ on } \partial \Omega \} $, where the first condition has to be understood in the distributional sense and the second condition in the sense of the trace in $H^{-1/2}(\partial \Omega )$. 
The dual space of a Banach space $V$ is always denoted by $ V^*$ and equipped with the standard norm; the duality pairing is denoted by $\langle\cdot, \cdot \rangle$ and the $L^2$-inner product by $(\cdot , \cdot )$.
The total variation of a function $E:\R\ra \R$ is given by 
$ \| E \|_{\text{TV}(0,T)}= \sup_{0<t_0<\ldots <t_n<T} \sum_{k=1}^N| E(t_{k-1})-E(t_k)|$
where the supremum is taken over all finite partitions of the interval $[0,T]$. 
We denote the space of all functions of bounded variations on $[0,T]$ by~$\BV$. 

Note that the total variation of a monotone decreasing nonnegative function only depends on the initial value, \textit{i.e.,}
\begin{align*}
\| E \|_{\text{TV}(0,T)} = \sup_{0<t_0<\ldots <t_n<T}\sum_{k=1}^N| E(t_{k-1})-E(t_k)| \leq E(0) - E(T) \leq E(0) \,.
\end{align*}

The symmetric part of a matrix is given by 
$\f A_{\sym}: = \frac{1}{2} (\f A + \f A^T)$  for $\f A \in \R^{d\times  d}$.
For the product of two matrices $\f A, \f B \in \R^{d\times d }$, we observe
 \begin{align*}
 \f A: \f B = \f A : \f B_{\sym}\,, \quad \text{if } \f A^T= \f A\, .
 \end{align*}
Furthermore, it holds
$ \f a\otimes \f b : \f A = \f a \cdot \f A \f b$ for
$\f a, \f b \in \R^d$, $\f A \in \R^{d\times d }$  and hence $ \f a \otimes \f a : \f A = \f a \cdot \f A \f a =  \f a \cdot \f A_{\sym} \f a$.
By $I$, we denote the identity matrix in $\R^{d\times d}$ and by $\R_+:=[0,\infty)$ the positive real numbers.

The following lemma provides the connection between the almost everywhere pointwise formulation of an inequality with the weak one. 
\begin{lemma}\label{lem:invar}
Let $f\in L^1(0,T)$ and $g\in L^\infty(0,T)$ with $g\geq 0$ a.e.~in $(0,T)$.
Then the two inequalities 
\begin{align*}
-\int_0^T \phi'(t) g(t) \de t  + \int_0^T \phi(t) f(t) \de t \leq 0 
\end{align*}
for all $\phi \in {\C}^1_c ((0,T))$ with $\phi \geq 0$ for all $t\in (0,T)$
and 
\begin{align}
g(t) -g(s) + \int_s^t f(s) \de \tau \leq 0 \quad \text{for a.e.~}t,\, s\in(0,T)\,\label{ineq2}
\end{align}
are equivalent. 
\end{lemma}
\begin{proof}
Since this is a rather standard lemma, I only shortly want to comment on the proof. For the if-direction, one may argues by  inserting an approximating sequence of the indicator function $\chi_{[s,t]}$ for $\phi$. To infer the only-if-direction, we sum up the second inequality for any partition $0<t_1<\ldots < t_N<T$ of $[0,T]$ to infer
\begin{align*}
\sum_{n=0}^{N-1}\phi(\xi_k)[g(t_{n+1})-g(t_n)] + \sum_{n=0}^{N-1}\phi(\xi_k )\int_{t_n}^{t_{n+1}}f(\tau)\de \tau \leq 0 \quad \text{with }\xi _n \in (t_n,t_{n+1})\,.
\end{align*}
Passing to the limit in the partition, gives the integral in the sense of Stieltjes (cf.~\cite[Chap.~8, Sec.~6]{natanson}). An integration-by-parts in the first term implies the first inequality in Lemma~\ref{lem:invar}. 
\end{proof}
Additionally, we use a lemma that provides the lower semi-continuity of convex functionals. 
\begin{lemma}\label{lem:lowcon}
Let $ A \subset \R^{d+1} $ be  a  bounded open set and $f: A \times \R^n \times \R^m \ra \R_+$ with $d, n, m \geq 1$, a measurable nonnegative function such that
 $ f ( \f y, \cdot, \cdot ) $ is lower semi-continuous on $\R^n\times \R^m$ for a.e.~$\f y \in A$, and $f$ is convex in the last entry. 
For sequences $ \{ \f u_k \}_{k\in \N} \subset L^1_{\loc}(A; \R^n)$ and $ \{ \f v_k \}_{k\in \N} \subset L^1_{\loc}(A; \R^m)$ as well as functions $\f u \in L^1_{\loc}(A; \R^n)$ and $ \f v \in L^1_{\loc}(A; \R^m)$ with
\begin{align*}
\f u _k \ra \f u \quad \text{a.e.~in }A \quad \text{and} \quad \f v_k \rightharpoonup \f v \quad \text{in }L^1_{\loc}(A ; \R^n) \,
\end{align*}
it holds 
\begin{align*}
\liminf_{k\ra \infty} \int_A f(\f y, \f u_k(\f y), \f v_k(\f y)) \de \f y \geq \int_A f ( \f y, \f u(\f y) , \f v(\f y)) \de \f y \,.
\end{align*}
\end{lemma}
The proof of this assertion can be found in~\cite{ioffe}. 

The following property of $\BV$-functions can for instance be found in~\cite{BV}. 
\begin{lemma}
Let $E:\R\ra \R$ be a function of bounded variation, $E\in\BV$. Then $E$ is continuous up to a countable subset of $(0,T)$ and the left- and right-limits are uniquely defined in every  interior point, \textit{i.e.},
\begin{align*}
E(t-) = \lim_{s\nearrow t} E(s) \quad E(t+)=\lim_{ s \searrow t} E(s) \quad\text{for all }t\in (0,T)
\end{align*}
and with one-sided limits at the end points. The usual choice are the so-called \glqq{}cadlag\grqq{}
(continuity a droit limit a gauche) representations by defining $E(t):=E(t-)$. Since we want to minimize the energy function $E$ and for a monotonously non-increasing function it always holds  that $ E(t+)\leq E(t-)$, we rather chose the right-continuous representation by defining~$E(t):=E(t+)$. 
\end{lemma}

\subsection{Definitions\label{sec:def}}
First we recall the Navier--Stokes  and Euler equations,
\begin{align}
\begin{split}
\t \f v +\di  ( \f v \otimes \f v) - \nu \Delta \f v + \nabla p = \f f \quad \text{and} \quad  \di \f v ={}& 0 \qquad \text{in }\Omega \times (0,T)\,,\\
\f v (0) ={}& \f v_0 \qquad \hspace*{-0.15cm}\text{in } \Omega \,,\\
\nu (I - \f n \otimes \f n ) \f v = 0 \quad \text{and} \quad 
\f n \cdot \f v   = {}& 0 \qquad \text{on }\partial \Omega\times (0,T) \,.
\end{split}\label{nav}
\end{align}

By writing the boundary conditions in this way, the system incorporates the Navier--Stokes system with no-slip conditions for $\nu>0$ and the Euler equations for $\nu=0$. Indeed, for $\nu>0$, the tangential and normal part of  the velocity field vanish such that this is equivalent to $\f v =0 $ on $\partial \Omega \times (0,T)$. For the case of $\nu =0$, \textit{i.e.}, no friction, only the normal component vanishes on the boundary. 
The underlying natural energy and dissipation spaces are given by   $ \mathbb{X}= L^\infty(0,T; \Ha)\cap L^2(0,T;\V)$ for $\nu>0$ and  $\mathbb{X}_0 =  L^\infty(0,T; \Ha) $ for $\nu =0$ and
the space of test-functions is given by 
 $\mathbb{Y}=  \mathbb{X}\cap L^2(0,T;H^2(\Omega)) \cap L^1(0,T;W^{1,\infty}(\Omega ))\cap H^1(0,T;(\Ha)^*) $ for $\nu>0$ and $\mathbb{Y}_0 = \mathbb{X}_0 \cap L^1(0,T;W^{1,\infty}(\Omega ) )\cap H^1 (0,T;(\Ha)^*)$ for $\nu=0$. 
 The space $\mathbb{Y}$ is chosen smooth enough such that the Stokes operator (for $\nu>0$) and the convection term map $\mathbb{Y} $ to $L^1(0,T;(\Ha)^*)$. 
The right-hand side $\f f$ is assumed to be in $\mathbb{Z}$, where 
$\mathbb{Z} := L^2 (0,T; H^{-1}(\Omega))\oplus H^1 (0,T;L^2(\Omega))$ for $\nu>0$ and  $\mathbb{Z}_0 := L^1(0,T;L^2(\Omega))$ for $\nu=0$.

We define the relative energy $\mathcal{R}:  \Ha \times \Ha \ra \R_+ $ by
\begin{subequations}\label{def:nav}
\begin{align}
\mathcal{R}(\f v|\vv) ={}&  \frac{1}{2}\| \f v - \vv \|_{L^2(\Omega)}^2 \,,
\intertext{and the system operator $\mathcal{A}_\nu:\mathbb{Y}\ra L^1(0,T;(\Ha)^*)$ via}
\langle \mathcal{A}_{\nu}(\vv) , \cdot \rangle ={}& \langle \t \vv + (\vv \cdot \nabla)\vv - \nu \Delta \vv  - \f f
, \cdot \rangle \,,\label{A}
\end{align}
\end{subequations}
which has to be understood in a weak sense, at least with respect to space. 

Note that the system operator does not include boundary conditions, since they are encoded in the underlying spaces. This may change for different boundary conditions. 

\begin{definition}\label{def:Kad}
We consider a form $\mathcal{K}: \Ha \ra [0,\infty]$ and define its domain by 
$\mathcal{D}(\mathcal{K}) := \{ \vv \in \mathbb{X}| \mathcal{K}(\vv) \in L^1(0,T)_+\}\,.
$ We assume that there is a fine enough  topology on $\mathcal{D}(\mathcal{K}) $ such that $\mathcal{K}$ is continuous  and $\C^1([0,T];\mathcal{C}_{c,\sigma}^\infty(\Omega;\R^d))$ is dense in $\mathcal{D}(\mathcal{K}) $ with respect to this topology. We always assume that $\mathcal{K}(0)=0$. 

The form $\mathcal{K}$ is called admissible for $\nu>0$, if the relative form $\mathcal{W}_{\nu}:
\mathbb{X}\times \mathbb{Y} \ra L^1(0,T)$ given by 
\begin{align}
\mathcal{W}_{\nu}(\f v | \vv) ={}& \nu \| \nabla \f v-\nabla \vv \|_{L^2(\Omega)}^2 - \int_\Omega  (( \f v-\vv) \cdot \nabla ) (\f v-\vv)  \cdot \vv \de \f x + \mathcal{K}(\vv) \mathcal{R}(\f v | \vv)  \label{Wnu}
\end{align}
is nonnegative for all $ \f v \in \mathbb{X}$ and all $\vv \in \mathbb{Y}\cap \mathcal{D}(\mathcal{K})$. 
Similarly, the form $\mathcal{K}_0$ is called admissible for $\nu=0$, if the relative form $\mathcal{W}_{0}:
\mathbb{X}_0\times \mathbb{Y} _0 \ra L^1(0,T)$ given by 
\begin{align}
\mathcal{W}_0(\f v | \vv) ={}& \int_\Omega ( \f v - \vv)^T \cdot ( \nabla \vv)_{\sym} ( \f v - \vv) \de \f x + \mathcal{K}_0(\vv) \mathcal{R}(\f v | \vv) \,.
\label{Wnull}
\end{align}
is nonnegative for all $ \f v \in \mathbb{X}_0$ and all $\vv \in \mathbb{Y}_0\cap \mathcal{D}(\mathcal{K}_0)$.
\end{definition}

\begin{example}
The standard example for a choice for $\mathcal{K}$ are the usual Serrin-type norms:
\begin{align}
\mathcal{K}(\vv)={}& \mathcal{K}_{\nu}^{s,r} (\vv)=c \| \vv\|_{L^{r}(\Omega)} ^{s} \quad \text{for } \frac{2}{s}+\frac{d}{r}= 1 \label{serrin}
\end{align}
with $r\in(d,\infty)$ and $s\in(2,\infty)$.
 Indeed 
H\"older's, Gagliardo--Nirenberg's, and Young's inequality provide the estimate for $\nu>0$
\begin{subequations}\label{West}
\begin{align}
\begin{split}
\left | \int_\Omega  (( \f v-\vv) \cdot \nabla ) (\f v-\vv)  \cdot \vv \de \f x \right |\leq{}&  \| \f v - \vv \|_{L^p(\Omega)} \| \nabla \f v - \nabla \vv \|_{L^2(\Omega)} \| \vv \|_{L^{2p/(p-2)}(\Omega)} \\ \leq{}& c_p \| \f v - \vv \|_{L^2(\Omega)} ^{(1-\alpha)}
 \| \nabla \f v - \nabla \vv \|_{L^2(\Omega)}^{(1+\alpha)}
  \| \vv \|_{L^{2p/(p-2)}(\Omega)} 
 \\ \leq{}& \frac{\nu}{2} \| \nabla \f v - \nabla \vv\|_{L^2(\Omega)}^2 + c \| \vv\|_{L^{2p/(p-2)}(\Omega)} ^{2/(1-\alpha)}
 \frac{1}{2} \| \f v - \vv \|_{L^2(\Omega)}^2 \,,
\end{split}\label{Kest}
\end{align}
where $\alpha $ is chosen according to Gagliardo--Nirenberg's inequality by $$\alpha = d (p-2)/2p\quad\text{for}\quad  d\leq 2p/(p-2)\,.$$ 
In the case of $\nu=0$, we may estimate 
\begin{align}
 \left ( ( \f v - \vv ) \otimes ( \f v - \vv) ; (\nabla \vv)_{\sym} \right )  \leq  2 | \vv |_{\C^{0,1}(\Omega)}
 \frac{1}{2}\| \f v - \vv \|_{L^2(\Omega)}^2 \,.
\end{align}
\end{subequations}

The estimate~\eqref{West} imply that $\mathcal{W}_\nu$ is nonnegative. 

\end{example}

\begin{remark}
\label{rem:W}
In contrast to previous publications, we want the form $\mathcal{K}$ to be general and not specifically chosen. This makes the solution concept of energy-variational solutions more selective and especially allows the vanishing viscosity limit in the proof of Theorem~\ref{thm:exist}. 
This makes no difference for the minimal energy-variational solution, since they turn out to be independent of the choice of $\mathcal{K}$. 
The form $\mathcal{K}$ is chosen in a way that the relative form
 $\mathcal{W}_\nu$  is nonnegative,  convex, and weakly-lower semi-continuous. 
 Indeed, since $\mathcal{W}_\nu$ is quadratic in $\f v$  and nonnegative, it is a standard matter to prove the  convexity of the mapping $\f v \mapsto \mathcal{W}_\nu(\f v | \vv) $. 
The mapping  $\f v \mapsto \mathcal{W}_\nu(\f v | \cdot)  $ is continuous in the strong topology in $\V$ and $\Ha$ for $\nu>0$ and $\nu=0$, respectively. Thus this mapping is weakly-lower semi-continuous (see for instance~\cite[Chap.~1, Cor.~2.2]{ekeland}). 

The assumption on the topology on $\mathcal{D}(\mathcal{K})$ and the continuity of $\mathcal{K}$ is of a technical nature. Instead of the choice $\mathcal{D}(\mathcal{K}) = L^2(0,T;L^\infty(\Omega) \cap \V )$ and $\mathcal{K}(\vv) = c \| \vv \|_{L^\infty(\Omega)}^2$ we rather chose the finer topology of $\mathcal{D}(\mathcal{K}) = L^2(0,T;\C_0 (\Omega) \cap \V )$ and $\mathcal{K}(\vv) =c  \| \vv \|_{\C (\Omega)}$ in the case $\nu>0$. 
It would also be possible to choose some intermediate separable space allowing jumps (see~\cite[Example~1.4.10 ]{RoubicekMeasure}).
For $\nu=0$, we choose instead of 
 $\mathcal{D}(\mathcal{K}) = L^1(0,T;W^{1,\infty}(\Omega) \cap \Ha )$ and $\mathcal{K}(\vv) = 2\| \nabla\vv \|_{L^\infty(\Omega)}$ the finer topology of $\mathcal{D}(\mathcal{K}) = L^1(0,T;\C^{0,1} ( \ov \Omega) \cap \Ha  )$ and $\mathcal{K}(\vv) =2 | \vv |_{\C^{0,1} (\ov\Omega)}$. These finer choices allow us to use the approximation property by density arguments. 
But also the case of the coarser topologies and associated $L^\infty$-norms could be made rigorous by an adapted method. 
\end{remark}

\begin{definition}[energy-variational solution]\label{def:diss} 
A tuplet $(\f v,\xi )$ is called an energy-variational solution, if $ (\f v,\xi  )\in  \mathbb{X}\times L^\infty(0,T)$  and  $\xi(t) \geq 0 $ for all~$t\in(0,T)$,  and for all admissible forms $\mathcal{K}:\mathbb{X} \supset \mathcal{D}(\mathcal{K})\ra L^1(0,T)_+$ according to Definition~\ref{def:Kad}, the relative energy inequality
\begin{multline}
\mathcal{R}(\f v( t+) | \vv(t) ) +\xi(t+)+ \int_s^t \left ( \mathcal{W}_\nu(\f v , \vv )   + \left \langle \mathcal{A}_\nu(\vv ) , \f v - \vv  \right \rangle - \mathcal{K}(\vv) (\mathcal{R}( \f v| \vv ) +\xi) \right )  \de \tau 
\\
\leq \mathcal{R}( \f v( s-) | \vv(s) )+ \xi (s-)  \label{relen}
\end{multline}
holds for all~$t$, $s\in [0,T]$ and for all $ \vv \in\mathbb{Y} \cap \mathcal{D}(\mathcal{K})$.
The initial value  $\f v(0)=\f v_0 $ is attained in the weak sense.
\end{definition} 
\begin{remark}[Properties of energy-variational solutions]\label{rem:semi}
By the Definition~\ref{def:diss} it is immediately clear that any energy-variational solution on an interval $[0,T]$ is also an energy-variational solution on any sub interval $(s,t)$ for all $s$, $t\in(0,T)$. Furthermore, let $(\f v^1, \xi^1)$ be an energy-variational solution on the interval $(0,t)$ and $(\f v^2, \xi^2)$ an energy-variational solution on the interval $(t,T)$ with $\f v^2(t) = \f v^1(t)$ and $\xi^2(t)\leq \xi^1(t)$. Then the concatenation of $(\f v^1, \xi^1)$ by $(\f v^2, \xi^2)$, \textit{i.e.}, the function $(\f v, \xi)$ given by 
\begin{align*}
\begin{cases}
(\f v (t), \xi(t)) = (\f v^1(t), \xi^1(t)) \quad &\text{for }t\in [0,t)\\
(\f v(t) , \xi(t) )= (\f v^2(t), \xi^2(t))  \quad & \text{for }t\in [t,T]
\end{cases}
\end{align*}
is again an energy-variational solution on $[0,T]$. This is the new key ingredient in comparisson to the previously introduced dissipative solutions~\cite{maxdiss}.

\end{remark}

\begin{remark}[Comparison to dissipative solutions]
Another difference of the proposed energy-var\-ia\-tion\-al solution framework in comparison to dissipative solutions lies in the definition of the relative form~$\mathcal{W}_{\nu}$. In dissipative solution concepts, the terms in the relative dissipation were only estimated from below by zero (see~\cite{lionsfluid} and~\cite{maxdiss}). The new insight is that these terms in $\mathcal{W}_\nu$ can be kept and don't have to be estimated. This also leads to the fact that the relative energy inequality is actually an equality for smooth solutions. Indeed in this case the energy inequality~\eqref{eninv} is an equality and thus also the relative energy inequality becomes an equality. 
Furthermore, for the regularity measure $\mathcal{K}$, we allow a family of functions. Finally, the introduction of the auxiliary variable $\xi$ allows to write down the relative energy inequality before applying the Gronwall argument and formulating the relative energy inequality on any subinterval of $[0,T]$. 




\end{remark}
\begin{corollary}[Refinement of dissipative solutions]
Let $(\f v , \xi) $ be an energy-variational solution according to Definition~\ref{def:diss} with $\xi (0) = 0$. Then $\f v $ is a dissipative solution, i.e. for an admissible $\mathcal{K}\geq 0$ according to Definition~\ref{def:Kad}, it holds that
\begin{align*}
\mathcal{R}(\f v( t) | \vv(t) ) + \int_0^t     \left \langle \mathcal{A}_\nu(\vv ) , \f v - \vv  \right \rangle \e^{\int_s^t\mathcal{K}(\vv)\de \tau} \de \tau 
\leq \mathcal{R}( \f v_0 | \vv(0) )  \e^{\int_0^t\mathcal{K}(\vv)\de s}\,,
\end{align*}
for a.e.~$t\in (0,T)$ and for all $ \vv \in \mathbb{Y}\cap \D $. Which is the definition according to Lions (see~\cite[Sec.~4.4]{lionsfluid}). This implies that in the case $\xi(0) = 0$, energy-variational solutions fulfill the so-called weak-strong uniqueness property. If a strong solution exists locally-in-time, every energy-variational solution coincides with this strong solution as long as the latter exists.
\end{corollary}

\begin{proof}
Let $(\f v , \xi) $ be a energy-variational solution according to Definition~\ref{def:diss}. 
 From the condition on the initial values, we infer
\begin{multline*}
\mathcal{R}(\f v( t) | \vv(t) ) +\xi(t)+ \int_0^t \left ( \mathcal{W}_\nu(\f v , \vv )   + \left \langle \mathcal{A}_\nu(\vv ) , \f v - \vv  \right \rangle - \mathcal{K}(\vv) (\mathcal{R}( \f v| \vv ) +\xi) \right )  \de \tau 
\\
\leq \mathcal{R}( \f v_0 | \vv(0) )  
\end{multline*}
for a.e.~$t\in (0,T)$ and for all $ \vv \in \mathbb{Y}\cap \mathcal{D}(\mathcal{K})$.
Gronwall's inequality and the property that $\xi \geq 0$ as well as $\mathcal{W}_\nu\geq 0$, implies the assertion. 
\end{proof}
\begin{corollary}[Additional regularity]\label{cor:addreg}
Let $(\f v , \xi) $ be an energy-variational solution according to Definition~\ref{def:diss}. 
Then the function $E(t):= 1/2\| \f v(t) \|_{L^2(\Omega)}^2 + \xi(t)$ is of bounded variation on $[0,T]$ and furthermore,   $ \f v \in H^1(0,T;(\V\cap W^{2,p}(\Omega))^*) \cap \C_w([0,T];\Ha)$
Note that the bounds of the associated estimates only depend on the initial values $\f v_0$ and $\xi(0)$ as well as the right-hand side $\f f\in \mathbb{Z}_\nu$.  This additional regularity allows tp formulate the relative energy inequality~\eqref{relen} everywhere in $[0,T]$ instead of almost everywhere. 

\end{corollary}
\begin{proof}
Let $E$ be given as above, by choosing $\vv=0$ in~\eqref{relen}, we find that
\begin{align}
E(t)-E(s) + \int_s^t \nu \| \nabla \f v\|^2_{L^2(\Omega)} - \langle \f f , \f v \rangle \de \tau \leq 0\,.\label{monotone}
\end{align}
for a.e.~$s$, $t\in(0,T)$. 
This implies that the function $t \mapsto  E(t) + \int_0^t \nu \|\nabla  \f v\|_{L^2(\Omega)}^2 - \langle \f f , \f v \rangle \de s$ is a monoton\-ous\-ly non-increasing function and thus a function of bounded variation~\cite{natanson}.

For $\f f \in \mathbb{Z}_\nu = L^2(0,T;H^{-1}(\Omega)) \oplus L^1(0,T;  L^2(\Omega)) $,  there exist  $\f f_1\in L^2(0,T; H^{-1}(\Omega)) $ and $\f  f_2 \in L^1(0,T;  L^2(\Omega))$ such that
we may estimate with H\"older's, Young's, and {Poincar\'e}'s inequality 
\begin{align}
\langle \f f , \f v \rangle \leq \frac{\nu }{2}\|\nabla \f v \|_{L^2(\Omega)}^2 + \frac{C}{2\nu }\| \f f _1 \|_{ H^{-1}(\Omega)} ^2 + \frac 1 2 \| \f f_2 \|_{L^2(\Omega)}  \left  (\| \f v \|_{L^2(\Omega)}^2+1\right ) \,.\label{rhs}
\end{align}
Inserting this into~\eqref{monotone} for $s=0$ and using $\xi \geq 0$, implies that 
\begin{multline*}
\| \f v (t) \|_{L^2(\Omega)}^2 + \int_0^t  \nu \| \nabla \f v\|^2_{L^2(\Omega)} \de s 
\\
\leq \| \f v_0\|_{L^2(\Omega)}^2+2 \xi (0)  + \int_0^t\frac{C}{\nu }\| \f f _1 \|_{ H^{-1}(\Omega)} ^2 +  \| \f f_2 \|_{L^2(\Omega)}  \left  (\| \f v \|_{L^2(\Omega)}^2+1\right ) \de s \,. 
\end{multline*}
Via Gronwall's lemma we infer $\f f$, $\f v_0$ and $\xi(0)$-dependent bounds on $\f v$ in $ \mathbb{X}$.

Since for $\f v \in \mathbb{X}$ and $\f f \in \mathbb{Z}$, the function
$\nu \|\nabla  \f v\|_{L^2(\Omega)}^2 - \langle \f f , \f v \rangle$ is integrable such that 
$\int_0^t \nu \|\nabla  \f v\|_{L^2(\Omega)}^2 - \langle \f f , \f v \rangle \de s$ is absolutely continuous and thus of bounded variation. 
The sum and difference of functions of bounded variation are known to be of bounded variation again (see~\cite[Chap.~8,Thm.~3]{natanson}), such that the first assertion follows, where the associated bound on $\| E\|_{\text{TV}(0,T)}$ again depends on $E(0)=1/2\| \f v_0 \|_{L^2(\Omega)}^2+ \xi(0)$ and $\f f$. 

Now choosing $t=T$ and $s=0$ in~\eqref{relen} as well as $\mathcal{K}(\vv) = 2 | \vv |_{\C^{0,1}(\Omega)}$ , we find
\begin{align}\label{timederi}
\begin{split}
- \int_0^T \langle \t \f v , \vv \rangle \de t ={}&-(\f v(T), \vv(T)) + ( \f v_0, \vv(0)) + \int_0^T ( \t \vv , \f v ) \de t \\ \leq{}& \int_0^T \nu ( \nabla \f v , \nabla \vv ) + ( ( \f v \cdot \nabla) \f v , \vv ) - \langle \f f , \vv \rangle - \xi \mathcal{K}(\vv) \de t 
\\
&- \left ( E (T)-E(0) + \int_0^T \nu \| \nabla \f v \| _{L^2(\Omega)}^2 - \langle \f f , \f v \rangle \de t \right )\,.
\end{split}
\end{align}
The second line on the right-hand side is known to be bounded and for the first line, we observe
\begin{multline*}
\int_0^T \nu ( \nabla \f v , \nabla \vv ) + ( ( \f v \cdot \nabla) \f v , \vv ) - \langle \f f , \vv \rangle - \xi \mathcal{K}(\vv) \de t \\
\leq \nu \| \nabla \f v \|_{L^2(\Omega\times (0,T))} \| \nabla \vv \|_{L^2(\Omega\times (0,T))} + \| \f v \|_{L^\infty(0,T;L^2(\Omega))}^2  \| \vv \|_{L^1(0,T;W^{1,\infty}(\Omega))} \\ + \| \f f \|_{\mathbb{Z}}\left (\nu\| \nabla \vv \|_{L^2(\Omega\times (0,T))} +\| \f v \|_{L^\infty(0,T;L^2(\Omega))}  \right )  +2  \| \xi \|_{L^\infty(0,T)} \| \vv \| _{L^1(0,T;\C^{0,1}(\ov\Omega))} 
 \,.
\end{multline*}
On the left-hand side of the  inequality~\eqref{timederi}, the definition of the weak-time derivative appears. On the right-hand side, the terms depending on $\f v$ and $\xi $ are bounded. 
Taking the supremum over all test functions, we observe that
\begin{align*}
\| \t \f v \| _{L^2(0,T;((\V\cap W^{2,p}(\Omega))^*)} = \sup_{ {\scriptsize\begin{matrix}
\vv \in  L^2(0,T;\V\cap W^{2,p}(\Omega)),\\ \| \vv \| _{L^2(0,T;\V\cap W^{2,p}(\Omega))}=1
\end{matrix}} }-  \langle  \partial _t \f v , \vv \rangle \leq C \quad \text{for }p>d \,. 
\end{align*}
From a standard lemma, we infer since $\Ha$ is reflexive that
\begin{align*}
\C_w([0,T];(\V\cap W^{2,p}(\Omega))^*) \cap L^\infty(0,T;\Ha) \subset \C_w([0,T];\Ha)\,
\end{align*}
and from this that $\f v \in \C_w([0,T];\Ha)$ (see~\cite{simon}). Thus the pointwise evaluation in~\eqref{relen} is well-defined. 

\end{proof}

\begin{definition}[Minimal energy-variational solution]\label{def:ad}
A tuplet $( \f v, \xi  )$ is called a minimal energy-var\-ia\-tio\-nal  solution,  if $ (\f v, \xi  ) $ is an energy-variational solution according to Definition~\ref{def:diss}  and if 
\begin{align*}
\frac{1}{2}\| \f v(t+)\|_{L^2(\Omega)}^2 +\xi(t+)= E(t+) \leq \bar{ E}(t+) = \frac{1}{2}\| \bar{\f v}(t+)\|_{L^2(\Omega)}^2 + \bar\xi(t+)
\end{align*}
for all $t\in [0,T]$ and all energy-variational solutions $ (\bar{\f v}, \bar\xi  )\in   \mathbb{X} $ according to Definition~\ref{def:diss} for a given initial value $\f v_0\in\Ha$ and right-hand side $\f f \in \mathbb{Z} $. 
\end{definition}

\begin{remark}[Selection criterion]
The proposed selection criterion relies on the insight that a physically relevant solution dissipates energy at the highest rate~(see~\cite{maximaldiss} or~\cite{dafermosscalar}). This leads to a minimized energy (compare the energy inequality~\eqref{eninv}, which is formally an equality). In a thermodynamical consistent system, the energy would be constant, but the maximized dissipation leads to a maximized entropy~(see~\cite{generic} for instance). This criterion was introduced as the entropy rate admissibility criterion~\cite{dafermosscalar}. 
There are different works on the entropy rate admissibility criterion applied to different systems. 
For instance, in the case of scalar conservation laws it was shown that this criterion coincides with the Oleinik-E condition and thus the usual entropy admissibility criterion for solutions with finitely many shocks~(see~\cite{dafermosscalar} or~\cite[Thm.~9.7.2]{dafermosbook} for the result). 
Since this criterion was proven to select the physically relevant solution in these scarcely available examples of nonlinear PDEs that are well understood, it may also does this for more involved systems (like the ones we consider here). 
One may chooses different selection criteria. Indeed, the local minimization should be equivalent to defining a global minimization, the energy-variational solution $(\f v ,\xi )$ minimizing  $\int_0^T \mathcal{E}\de t = \int_0^T \frac{1}{2}\| \f v \| _{L^2(\Omega)}^2 + \xi \de t $. The local minimizer is also a global minimizer by construction and since the local minimizer is in the admissible set of the global optimization problem both minima are equal up to choosing the correct representative in the associated equivalence class of a.e.~equal functions. Here, I stick to the local selection principle since it follows the physical principle of maximal dissipation more closely and the proof of the main result is more obvious. 

\end{remark}
\begin{remark}
We note that the Definition~\ref{def:ad} is well defined. Due to the additional regularity deduced in Corollary~\ref{cor:addreg}, we know that the function $E$ for an energy-variational solution $(\f v , \xi)$ according to Definition~\ref{def:diss} is a $\BV$ function and thus the limit $\lim _{s\searrow t} E(s) = E(t+) $ exists and is unique for all $t\in[0,T)$. 
That a solution fulfilling this criterion exists and is unique is provided in Theorem~\ref{thm:maxdiss}. 
Note that due to the inequality~\eqref{monotone}, only negative jumps with $\lim _{s\searrow t} E(s) = E(t+) \leq \lim _{s \nearrow t} E(s) = E(t-)$ are allowed. The increasing contribution to $E$ are only due to the right-hand side $\f f$ and thus, by construction absolutely continuous. 
Note also that the pointwise minimization in the case of $t=0$ immediately implies that $E(0) = 1/2 \| \f v_0\|_{L^2(\Omega)}^2 $, since the relative-energy inequality is automatically fulfilled for $t=s=0$.

\end{remark}

\begin{definition}[weak solution]\label{def:weak} 
A function $\f v$ is called a weak solution with the strong energy inequality, if $ \f v \in \mathbb{X}$ fulfills the  strong energy inequality
\begin{align}
\frac{1}{2}\| \f v \|_{L^2(\Omega)}^2 \Big|_s^t +\int_s^t \nu \| \nabla \f v \|_{L^2(\Omega)}^2 \de \tau \leq \int_s^t \langle \f f , \f v \rangle \de \tau \quad \text{for a.e.~}s<t\in (0,T)\,\label{eninv}
\end{align}
and the weak formulation
\begin{align}
-\int_0^ T \int_\Omega \f v \t \f \varphi  \de \f x \de t + \int_0^T \int_\Omega \left ( \nu \nabla \f v : \nabla \f \varphi - ( \f v \otimes  \f v ) :\nabla  \f \varphi \right ) \de \f x \de t = \int_0^T \langle \f f , \f \varphi \rangle \de t + \int_\Omega \f v_0 \cdot\f \varphi(0) \de \f x 
\label{weakv}
\end{align}
for every $\f \varphi \in\C^1_c([0,T)) \otimes \mathcal{C}_{c,\sigma}^\infty(\Omega;\R^{d})$.
\end{definition}

\subsection{Main results\label{sec:main}}
The main results of the paper at hand are the following.

 \begin{proposition}\label{prop:diss}
 Let $\f v $ be a weak solution according to Definition~\ref{def:weak}. 
 Then $\f v$ with $\xi = 0$ is an energy-variational solution according to Definition~\ref{def:diss}.

 \end{proposition}

  \begin{proposition}\label{prop:homo}
 Let $(\f v, \xi ) \in  \mathbb{X}\cap L^\infty(0,T) $
 be an energy-variational  solution solution according to Definition~\ref{def:diss}. 
 Assume that the regularity measure $\mathcal{K}$ is homogeneous of rank one, \textit{i.e.,} 
$\mathcal{K}(\alpha \vv) = | \alpha | \mathcal{K}(\vv) $ for all $\alpha\in\R$.  
 Then it holds that 
 \begin{subequations}
 \begin{align}
 \left (\frac{1}{2}\| \f v\|_{L^2(\Omega)}^2+\xi \right ) \Big|_q^r + \int_q^r \nu \| \nabla \f v\|_{L^2(\Omega)}^2 - \langle \f f , \f v \rangle \de \tau \leq {}&0\,
 \quad\text{for a.e.~} q<r\in (s,t)
 \label{eneq1}
 \intertext{and}
 \int_s^t \left (  \nu \left ( \nabla \f v , \nabla \vv\right ) - \left ( \f v \otimes \f v ; \nabla \vv\right ) - \left ( \f v , \t \vv \right ) - \langle \f f , \vv \rangle\right )  \de \tau
 \in {}& B \left ( 0 ,\int_s^t \mathcal{K}(\vv ) \xi \de \tau \right )   \,\label{weakeq1}
 \end{align}
 for all $\vv\in\C^1_c((s,t);\C^\infty_{0,\sigma}(\Omega))\cap \D$ and all $s<t\in(0,T)$. Here $B \left ( 0 ,r\right )$ denotes the ball in $\R $ around $0$ with radius $r$. 

  \end{subequations}
 
 \end{proposition}
  \begin{proposition}\label{prop:equation}
 Let $(\f v, \xi ) \in  \mathbb{X}\cap L^\infty(0,T) $
 be an energy-variational  solution solution according to Definition~\ref{def:diss}. 
 Assume that $\xi (\tau) = 0 $ for a.e.~$\tau \in (s,t)$, with 
 $s<t\in (0,T)$. 
 Then it holds that 
 \begin{subequations}
 \begin{align}
 \frac{1}{2}\| \f v\|_{L^2(\Omega)}^2\Big|_q^r + \int_q^r \nu \| \nabla \f v\|_{L^2(\Omega)}^2 - \langle \f f , \f v \rangle \de \tau \leq {}&0\,
 \quad\text{for a.e.~} q<r\in (s,t)
 \label{eneq}
 \intertext{and}
 \int_s^t\left (  \nu \left ( \nabla \f v , \nabla \vv\right ) - \left ( \f v \otimes \f v ; \nabla \vv\right ) - \left ( \f v , \t \vv \right ) - \langle \f f , \vv \rangle\right )  \de \tau
 ={}& 0 \,\label{weakeq}
 \end{align}
 for all $\vv\in\C^1_c((s,t);\C^\infty_{0,\sigma}(\Omega))$.
  \end{subequations}
   \end{proposition}
\begin{remark}
If the auxiliary variable $\xi$ vanishes, then the Navier--Stokes equation is fulfilled in the weak sense. 
 
In case that $\mathcal{K}$ is homogeneous of rank one, the error in the weak formulation can be characterized by the auxiliary variable $\xi$. 
 
\end{remark}
\begin{proposition}\label{prop:equi}
The relative energy inequality~\eqref{relen} of Definition~\ref{def:diss} can be equivalently written as 
\begin{itemize}
\item the locally-in-time version
\begin{multline}
\frac{\de }{\de t} \left [ \mathcal{R}(\f v(t)|\vv(t)) + \xi (t) \right ] + \mathcal{W}(\f v(t)| \vv(t)) + \langle\mathcal{A}(\vv(t)), \f v(t) - \vv (t)\rangle \\
- \mathcal{K}(\vv(t)) \left [ \mathcal{R}(\f v (t)| \vv(t) ) + \xi(t) \right ] \leq 0 \label{locally}
\end{multline}
for a.e.~$t\in (0,T)$ and all $\vv \in\mathbb{Y}\cap \D$. 

\item the reduced locally-in-time version
\begin{multline}
\t \left ( E(t) - ( \f v(t) , \vv) \right) + \nu \left ( \nabla \f v(t) ; \nabla \f v (t)- \nabla \vv\right ) + \left ( \f v(t) \otimes \f v(t) ; \nabla \vv \right ) - \langle \f f(t) , \f v (t)- \vv \rangle \\ + \mathcal{K}(\vv) \left ( \frac{1}{2}\| \f v(t) \|_{L^2(\Omega)}^2 - E (t)\right ) {} \leq 0 \,.\label{reduced}
\end{multline}
for a.e.~$t\in (0,T)$  and all $\vv \in \C^\infty_{0,\sigma}(\Omega) $, where $E = 1/2\| \f v\|_{L^2(\Omega)}^2 + \xi $. 
\end{itemize}
The time derivatives of $E$ has to be understood in the usual $\BV$ sense, the sense of Radon measures and the time-derivative of $\f v$ in the weak sense. 

\end{proposition}

\begin{theorem}\label{thm:exist}
Let $ \Omega \subset \R^d$ for $d\geq 2$ be a {bounded} Lipschitz domain, $\nu \geq  0$. Let $\mathcal{R}$, $\mathcal{W}_\nu$, $\mathcal{K}$, and $\mathcal{A}_\nu$ be given as above in~\eqref{def:nav} and let $\mathcal{K}$ fulfill Definition~\ref{def:Kad}. 
 
Then there exists at least one energy-variational solution $\f v \in \mathbb{X}$ to every  $\f v_0\in \Ha$ and $\f f \in \mathbb{Z} $  in the sense of Definition~\ref{def:diss}.
For every energy-variational solution $(\f v , \xi)$, we define as above $E= 1/2\| \f v \|_{L^2(\Omega)}^2 + \xi $. Then, the 
 set of solutions $\mathcal{S}(\f v _0 , \f f) \subset \mathbb{X} \times \BV$ consists of the pairs $(\f v ,E)$, with $(\f v , \xi)$ being an energy-variational solution  according to Definition~\ref{def:diss} to a given initial-value $\f v_0\in \Ha$ and right-hand side $\f f \in \mathbb{Z}$. 
The set $\mathcal{S}(\f v _0 , \f f)$  is convex and weakly$^*$-closed and thus compact in the weak$^*$ topology of $\mathbb{X}\times \BV $.

Additionally, the set-valued mapping $\mathcal{S}: \Ha \times \mathbb{Z} \ra \mathbb{X}\times \BV$, which maps $( \f v_0 , \f f)$ to the solution set consisting of  elements $(\f v, E)\in \mathcal{S}(\f v _0, \f f)$ is continuous in the set valued sense, \textit{i.e.,} if $( \f v ^n_0 , \f f^n ) \ra ( \f v _0 , \f  f ) $ in $ \Ha \times \mathbb{Z}$, then the associated solutions sets $\mathcal{S}(\f v_0^n , \f f^n) $ converge to $\mathcal{S}(\f v_0, \f f)$ in the Kuratowski sense with respect to the topology induced by the weak$^*$-convergence in $\mathbb{X} \times \BV$. 

This means, that to every element $(\f v, E)\in \mathcal{S}(\f v _0, \f f)$ and every  sequence $( \f v ^n_0 , \f f^n ) \ra ( \f v _0 , \f  f ) $, we may construct a sequence $\{( \f v^n , E^n)\}$ such that 
\begin{align*}
( \f v^n , E^n)
\ra  (\f v, E) \quad \text{in }\mathbb{X}\times \BV\,.
\end{align*}
(which is referred to as lower semi-continuity), and if there exists a sequence  $( \f v ^n_0 , \f f^n ) \ra ( \f v _0 , \f  f ) $ and a sequence $ ( \f v^n , E^n) \stackrel{*}{\rightharpoonup} (\f v, E) $ with $(\f v ^n ,E^n) \in \mathcal{S}(\f v_0^n, \f f^n)$ then  $(\f v, E)\in \mathcal{S}(\f v _0, \f f)$  (which is referred to es upper semi-continuity). 
\end{theorem}
\begin{remark}
In the case of $d=2,\,3$ or $4$, the existence of weak solutions to the Navier--Stokes equations is well known (see for instance~\cite{temam}). Due to Proposition~\ref{prop:diss}, this also proves the existence of energy-variational solutions.
The new result of the preceding theorem is expanding the existence of energy-variational solutions to any space dimension. 

\end{remark}

\begin{theorem}[Main result]\label{thm:maxdiss}
 Let $ \Omega \subset \R^d$ for $d\geq 2$ be a bounded Lipschitz domain, $\nu \geq  0$. Let $\mathcal{R}$, $\mathcal{W}_\nu$, $\mathcal{K}$, and $\mathcal{A}_\nu$ be given as above in~\eqref{def:nav} and let $\mathcal{K}$ fulfill Definition~\ref{def:Kad}.

  Then there exists a unique minimal  energy-variational solution $\f v \in \mathbb{X}$ to every  $\f v_0\in \Ha$ and $\f f \in \mathbb{Z} $
  in the sense of Definition~\ref{def:ad}.  
  Additionally, it holds that $\xi (t) =0 $ for a.e.~$t\in(0,T)$, which implies by Proposition~\ref{prop:equation} that the admissible solution is actually a weak solution according to Definition~\ref{def:weak}.

  The minimal energy-variational solution depends continuously on the initial datum and the right-hand side in the {following sense: If}  
$( \f v ^n_0 , \f f^n ) \ra ( \f v _0 , \f  f ) $ in $ \Ha \times \mathbb{Z}$, then  to every $n\in \N$, there exists a minimal energy-variational solution $\f v^n \in \mathbb{X}$ and it holds 
$ \f v^n \ra  \f v $ in $L^2(0,T;\Ha)$.

\end{theorem} 
\begin{remark}
The  minimal energy-variational solution fulfills the semiflow property by Remark~\ref{rem:semi} and the uniqueness.

\end{remark}

\section{Proofs of the main theorems\label{sec:proof}}
\subsection{Energy-variational and weak solutions}
First, we show that the velocity $\f v$ of a weak solution  is  an energy-variational solution. 

 \begin{proof}[Proof of Proposition~\ref{prop:diss}]
 
 Let $\f v $ be a weak solution to the Navier--Stokes and Euler equations~\eqref{nav} with strong energy inequality for~$\nu\geq0$.

For a test function $\vv\in  \mathbb{Y} $, we find by testing the system operator $\mathcal{A}_\nu(\vv)$ by $\phi \vv$ with $\phi \in \C^1_c([0,T))$ and standard calculations that
\begin{multline}
\int_0^T \phi \left \langle\mathcal{A}_\nu(\vv),\vv  \right\rangle   \de t =\\ -\int_0^T \phi '   \frac{1}{2} \| \vv (t) \|^2_{L^2(\Omega)} \de t   + \int_0^T \phi\left ( \nu  \| \nabla \f v \|^2_{L^2(\Omega)} - \langle \f f , \vv \rangle\right )  \de t - \phi(0) \frac{1}{2}\| \vv (0) \|_{L^2(\Omega)}^2 \,. \label{eninvv}
\end{multline}
Testing again the system operator $\mathcal{A}_\nu(\vv)$ by $\phi\f v$ and choosing $\f \varphi $ to be $\phi \vv$ in~\eqref{weakv}   with $\phi \in \C^1_c([0,T))$ (or approximate it appropriately), we find
\begin{multline}
- \int_0^T \phi' \int_\Omega \f v \cdot \vv \de \f x \de t + \int_0^T \phi\left ( \int_\Omega \left ( 2 \nu \nabla \f v : \nabla \vv - ( \f v \otimes  \f v) : \nabla  \vv + ( \vv \cdot \nabla) \vv \cdot \f v \right ) \de \f x  \right ) \de t \\= \int_0^T \phi \left\langle  \mathcal{A}_\nu(\vv), \f v\right \rangle  \de t + \phi(0) \int_\Omega \f v_0 \cdot \vv(0) \de \f x+ \int_0^T \phi\langle\f  f , \vv + \f v \rangle \de t  \,.\label{eninboth}
\end{multline}
Reformulating~\eqref{eninv} by Lemma~\ref{lem:invar}, adding~\eqref{eninvv}, as well as subtracting~\eqref{eninboth}, let us deduce that 
\begin{multline}
-\int_0^T \phi'\frac{1}{2} \| \f v - \vv \|_{L^2(\Omega)}^2
 \de t  +\nu \int_0^T \phi \| \nabla \f v - \nabla \vv \|^2_{L^2(\Omega)}\de t - \phi(0) {\frac{1}{2}\| \f v_0 -\vv(0) \|_{L^2(\Omega)}^2} \\  -  \int_0^T\phi \left (  \int_\Omega  \left (  ( \vv \cdot \nabla) \vv \cdot \f v- ( \f v \otimes \f v) : \nabla \vv \right )  \de \f x 
  \right )  \de t + \int_0^T\phi \left  \langle \mathcal{A}_\nu(\vv),\f v- \vv\right \rangle  \de t\leq 0 \,\label{secrelen}
\end{multline}
for all $\phi\in{C}^1_c([0,T])$ with $\phi\geq 0$ a.e.~on $(0,T)$.
%
We adopt some standard manipulations using the skew-symmetry of the convective term in the last two arguments and the fact that $\f v$ and $\vv$ are divergence free, to find
\begin{align}
-\int_\Omega  \left (  ( \vv \cdot \nabla) \vv \cdot \f v- ( \f v \otimes \f v) : \nabla \vv \right )  \de \f x ={}& -
\int_\Omega \left (  (\f v \cdot \nabla) ( \f v - \vv) \cdot \vv  + ( \vv \cdot \nabla) \vv \cdot ( \f v- \vv) 
\right ) 
\de \f x \notag
\\ ={}&  -\int_\Omega  (( \f v-\vv) \cdot \nabla ) (\f v-\vv)  \cdot \vv \de \f x\,\notag
\intertext{for $\nu>0$ and}
-\int_\Omega  \left (  ( \vv \cdot \nabla) \vv \cdot \f v- ( \f v \otimes \f v) : \nabla \vv \right )  \de \f x ={}&  \int_\Omega (\f v - \vv )^T \cdot \left ( \nabla \vv\right )_{\sym} ( \f v- \vv)\de \f x \\
{}& \int_\Omega ( ( \f v - \vv )\otimes \vv ) : \nabla \vv \de \f x  +\int_\Omega (\vv\cdot\nabla) \vv\cdot \left (  \f v - \vv \right )\de \f x\notag  \\
 ={}&  \int_\Omega (\f v - \vv )^T \cdot \left ( \nabla \vv\right )_{\sym} ( \f v- \vv) \de \f x \,\label{convcalc}
\end{align}
for $\nu =0$. 
Inserting this into~\eqref{secrelen}, adding as well as subtracting $\mathcal{K}_\nu(\vv)\mathcal{R}(\f v| \vv) $
, we conclude 
\begin{multline*}
-\int_0^T \phi '\frac{1}{2 }\| \f v(t) - \vv(t) \|_{L^2(\Omega)}^2
\de t 
\\ +\int_0^T \phi  \left [
\mathcal{W}_\nu (\f v | \vv)
+\left \langle \mathcal{A}_\nu(\vv),\f v - \vv  \right \rangle 
- \mathcal{K}(\vv) \frac{1}{2 }\| \f v - \vv \|_{L^2(\Omega)}^2
\right ]
 \de s \leq 0 \,
\end{multline*}
for every function $\vv\in \mathbb{Y}\cap \D $ and   all $\phi\in{C}^1_c((0,T))$ with $\phi\geq 0$ a.e.~on $(0,T)$.
Lemma~\ref{lem:invar} and choosing $\xi=0 $
implies~\eqref{relen}. 
\end{proof}

\begin{proof}[Proof of Proposition~\ref{prop:homo}]

Now, we assume that $(\f v, \xi )\in  \mathbb{X}_{\nu}\times L^\infty (0,T)$ is an energy-variational solution according to Definition~\ref{def:diss}. 
Firstly, we observe that the relative energy inequality~\eqref{relen} with $\vv=0 $ gives the energy inequality~\eqref{eneq1}.

Secondly, we infer from the relative energy inequality~\eqref{relen} and Lemma~\ref{lem:invar} that
\begin{align*}
-\int_s^t \phi' [\mathcal{R}(\f v | \vv)+\xi]  \de \tau + \int_s^t \phi \left [ \mathcal{W}_\nu(\f v| \vv) + \left \langle \mathcal{A}_\nu(\vv) , \f v - \vv \right \rangle - \mathcal{K}_\nu(\tu) (\mathcal{R}(\f v | \vv) +\xi) \right ] \de \tau \leq 0 
\,
\end{align*}
for all $\phi\in{C}^1_c(s,t)$ with $\phi\geq 0$ a.e.~on $(s,t)$. 
The Definition of $\mathcal{W}_\nu$ implies
\begin{subequations}\label{redineq}
\begin{multline}
-\int_s^t \phi '[\mathcal{R}(\f v| \vv) +\xi] \de \tau
  \\+ \int_s^t\phi \left [ \nu \| \nabla \f v - \nabla \vv \|_{L^2(\Omega)}^2 - \left ( ( (\f v -\vv) \cdot \nabla )( \f v - \vv ) , \vv \right )  + \left \langle \mathcal{A}_\nu(\vv) , \f v- \vv \right \rangle - \mathcal{K}(\vv) \xi \right ] \de \tau 
\leq  0 \,
\end{multline}
for $\nu>0$ and 
\begin{multline}
- \int_s^t \phi' [\mathcal{R}(\f v| \vv) +\xi] \de t 
\\+ \int_s^t \phi 
\left [ \left ( ( \f v- \vv) \otimes ( \f v -\vv) , (\nabla \vv)_{\sym}\right )   + \left \langle \mathcal{A}_\nu(\vv) , \f v- \vv \right \rangle- \mathcal{K}(\vv) \xi  \right ] \de \tau  \leq  0  \,
\end{multline}
\end{subequations}
for $\nu = 0$.  
For the system operator $\mathcal{A}_\nu$, we find 
\begin{multline*}
 \int_s^t \phi  \left \langle \mathcal{A}_\nu(\vv), \f v - \vv \right \rangle \de \tau = \int_s^t \phi' \frac{1}{2} \| \vv\|_{L^2(\Omega)}^2 \de \tau   
  \\+ \int_s^t \phi \left [ \left ( \t \vv , \f v \right ) + \nu \left ( \nabla \vv , \nabla \f v - \nabla \vv \right ) - \left ( (\vv \cdot \nabla) ( \f v -\vv ) , \vv \right ) - \langle \f f , \f v - \vv \rangle\right ] \de \tau \,
\end{multline*}
for all $\phi \in \C_c^1((s,s))$ with $\phi \geq 0$ a.e.~on $(s,t)$.
Inserting this into~\eqref{redineq}, we may deduce
\begin{multline}
- \int_s^t \phi ' \left [\frac{1}{2} \| \f v \|_{L^2(\Omega)}^2 - \left ( \f v , \vv \right ) +\xi \right ] \de \tau 
 \\+ \int_s^t\phi \left [ \nu \left ( \nabla \f v, \nabla \f v - \nabla \vv \right ) + \left ( \f v \otimes   (\f v-\vv), \nabla \vv \right ) + (\t \vv , \f v ) - \langle \f f, \f v - \vv \rangle - \mathcal{K }(\vv) \xi \right ] \de \tau \leq  0\,.\label{before}
\end{multline}
 Again the skew-symmetry of the trilinear form in the last two entries is used.  Choosing $  \vv = \alpha \tu$ and multiplying the inequality by $1/\alpha$ for $\alpha >0$, we find
\begin{multline}
\frac{1}{\alpha}\left (  
- \int_s^t \phi' \left [\frac{1}{2}\| \f v \|_{L^2(\Omega)}^2+ \xi \right ] \de \tau  
+ \int_s^t \phi \left [\nu  \| \nabla \f v \|_{L^2(\Omega)}^2 - \langle \f f , \f v \rangle \right ] \de \tau \right ) 
\\
+ \int_s^t \left ( \f v , \t(\phi \tu)\right ) \de \tau  - \int_s^t  \phi \left [\nu \left ( \nabla \f v , \nabla \tu\right )  -  \left ( ( \f v \otimes \f v ), \nabla \tu \right )  - \langle \f f , \tu \rangle \right ]    \de \tau  
\leq  \int_s^t \phi \mathcal{K}(\tu) \xi \de \tau \,.\label{weaqineq}
\end{multline}
 Note that the term $((\f v \cdot \nabla) \vv , \vv )$ vanishes since $\f v $ is solenoidal. Additionally, it is used that $\mathcal{K}$ is homogeneous of rank one. 
For $\alpha \ra \infty$ the first line in~\eqref{weaqineq} vanishes and in the resulting inequality we may observe that $\tu$ occurs linearly such that by inserting $\tu$ as well as $-\tu$, we receive two  inequalities,
\begin{multline*}
-  \int_s^t \phi  \mathcal{K}(\tu) \xi \de \tau
\\
\leq  \int_s^t \left ( \f v ,  \t (\tu\phi)\right ) \de t  - \int_s^t  \phi \left [\nu \left ( \nabla \f v , \nabla \tu\right )  -  \left ( ( \f v \otimes \f v ), \nabla \tu \right ) - \langle \f f , \tu \rangle \right ] \de \tau
\\
\leq  \int_s^t \phi \mathcal{K}(\tu) \xi \de \tau \,. 
\end{multline*}
By defining $\vv = -\phi  \tu$, we may observe the  formulation~\eqref{weakeq1}.


 \end{proof}
 \begin{proof}[Proof of Proposition~\ref{prop:equation}]
 This proof is very similar to the previous proof. All arguments in the previous proof up to the inequality~\eqref{before} are independent of the rank-1-homogeneity of $\mathcal{K}$. Thus choosing $\vv=\alpha \tu $ in~\eqref{before} and multiplying by $1/\alpha$ implies~\eqref{weaqineq} with $\xi \equiv 0$. Thus for $\alpha\ra \infty$, we infer due to the linearity of the test function $\tu$ the relation~\eqref{weakeq1} with $\xi \equiv 0$, which is nothing else than~\eqref{weakeq}. 
 \end{proof}

 \begin{proof}[Proof of Proposition~\ref{prop:equi}]
We may consider the relative energy inequality for $t=t+h$ and $s = t $ and multiply the inequality by $1/h$. Taking the limit $h\searrow 0$ in the resulting inequality, we infer that
\begin{multline*}
\frac{1}{h}
 \int_t^{t+h} \left [ \mathcal{W}_\nu(\f v , \vv )   + \left \langle \mathcal{A}_\nu(\vv ) , \f v - \vv  \right \rangle - \mathcal{K}(\vv) (\mathcal{R}( \f v| \vv ) +\xi) \right ]  \de \tau  \ra  \\\mathcal{W}_\nu(\f v(t) , \vv (t))   + \left \langle \mathcal{A}_\nu(\vv(t) ) , \f v(t) - \vv(t)  \right \rangle - \mathcal{K}(\vv(t)) [\mathcal{R}( \f v(t)| \vv(t) ) +\xi(t)]
\end{multline*}
for a.e.~$t\in(0,T)$ since all appearing terms are Lebesgue integrable. 

Since $E= \frac{1}{2}\| \f v \|_{L^2(\Omega)}+ \xi$ is a $\BV$ function, its derivative is a Radon measure (see~\cite[Chap.~8, Sec.~8]{natanson}, or~\cite[Thm.~2.13]{BV}). 
For the mixed term in the relative energy, we infer by the product rule for weak derivatives that 
\begin{align}
\lim_{h \searrow 0 } \frac{1}{h} \left [\left ( \f v (t+h), \vv (t+h)\right ) - \left ( \f v (t),\vv(t) \right )\right ] = \frac{\de }{\de t } \left (  \f v(t), \vv(t) \right ) = \left ( \t \f v(t) , \vv(t) \right ) + \left ( \f v(t) , \t \vv (t) \right ) \label{product}\,
\end{align}
for a.e.~$t\in(0,T)$. 
The derivative is well defined due to the additional regularity deduced in Corollary~\ref{cor:addreg} for the time derivative of $\f v$ and the choice of test functions $\vv$. Thus, for a.e.~$t\in(0,T)$, we may identify 
\begin{align*}
 \lim _{h \searrow 0 } \frac{1}{h}\left [ \mathcal{R}(\f v (t+h) | \vv(t+h) + \xi (t+h) - \mathcal{R}(\f v(t)| \vv(t) )- \xi (t)  \right ] = \frac{\de }{\de t } \left[ \mathcal{R}(\f v(t)| \vv(t) )+  \xi (t)  \right ]
\end{align*}
which is well-defined a.e.~in $(0,T)$ in the sense of Radon measures (see~\cite{natanson,BV}). Note that the pointwise inequality is equivalent to the inequality in the distributional sense, \textit{i.e.,} 
\begin{align*}
- \int_0^T \phi'\left [ \mathcal{R}(\f v(t)| \vv(t) )+ \xi (t)  \right ] \de t + \int_0^T \phi\left [ \mathcal{W}_\nu(\f v(t) , \vv (t))   + \left \langle \mathcal{A}_\nu(\vv(t) ) , \f v(t) - \vv(t)  \right \rangle \right] \de t \\ -\int_0^T \phi  \mathcal{K}(\vv(t)) [\mathcal{R}( \f v(t)| \vv(t) ) +\xi(t)] \de t \leq 0 \,
\end{align*}
for all $\phi \in \C_c^\infty(0,T)$ with $\phi \geq 0$ on $(0,T)$. 

Inserting the definition of $E$, $\mathcal{A}_\nu$, and $\mathcal{W}_\nu $ and using the product rule for the weak derivative~\eqref{product}, we may infer the reduced formulation~\eqref{reduced} for $\vv = \vv(t)$. Then it is possible to drop the time-dependency in the test function.


 \end{proof}

 
 \subsection{Existence of energy-variational solutions}
 In order to prove the existence of energy-variational solutions, we pass to the limit in the relative energy inequality. Therefore, we do not need any strong compactness arguments, which are essential in existence proofs for weak solutions to nonlinear PDEs. 
The formulation of the relative energy inequality allows to pass to the limit only relying on weakly-lower semi-continuity of the associated functionals and Helly's selection principle.

 \begin{proof}[Proof of Theorem~\ref{thm:exist}] 
The proof is based on the usual Galerkin approximation together with standard weak convergence techniques. We divide  the proof in different steps.

\textit{Step 1, Galerkin approximation:}
Since the space $\V$ is separable and the space of smooth solenoi\-dal functions with compact support, $\mathcal{C}_{c,\sigma}^\infty(\Omega;\R^d)$, is dense in $\V$, 
there exists a Galerkin scheme of $\V$, \textit{i.e.}, $\{ W_n\}_{n\in\N}$ with $\clos_{\V} ( \lim_{n\ra \infty} W_n ) = \V$.
 Let $P_n : \Ha \longrightarrow W_n$ denote the
$\Ha$-orthogonal projection onto $W_n$.
The approximate problem is then given as follows: Find {an absolutely continuous solution $\f v^n$ with  $\f v^n(t)\in W_n$ for all $t\in[0,T]$} solving the system
\begin{align}
\left ( \t \f v ^n + ( \f v ^n  \cdot \nabla) \f v ^n , \f w \right ) + \nu \left ( \nabla \f v^n ;\nabla \f w \right ) = \left \langle \f f , \f w\right \rangle \,, \quad \f v ^n(0) = P_n \f v_0 \quad {\text{for all }\f w \in W_n} \,.\label{vdis}
\end{align}

A classical existence theorem (see Hale~\cite[Chapter I, Theorem 5.2]{hale}) provides, for every $n\in\N$, the existence of a maximal extended solution to the above approximate problem~\eqref{vdis} on an
interval~$[0,T_n)$ in the sense of Carath\'e{}odory.

\textit{Step 2, \textit{A priori} estimates:}
%
It can be deduce that $T_n=T$ for all $n\in\N$, 
 if the solution undergoes no blow-up. With the standard \textit{a priori} estimates, we can exclude blow-ups and thus deduce global-in-time existence. \label{sec:exloc}
Testing~\eqref{vdis} by $\f v^n$, we derive the standard energy estimate
\begin{align}
\frac{1}{2} \frac{\de }{\de t} \| \f v^n \|_{L^2(\Omega)}^2 + \nu  \| \nabla \f v ^n\|_{L^2(\Omega)}^2  = 
 \langle \f f , \f v^n \rangle \,.\label{energydis}
\end{align}
For $\f f \in \mathbb{Z} = L^2(0,T;H^{-1}(\Omega)) \oplus L^1(0,T;  L^2(\Omega)) $ for $\nu>0$, the right-hand side can be estimated appropriately. Indeed, there exist two functions $\f f_1\in L^2(0,T; H^{-1}(\Omega)) $ and $\f  f_2 \in L^1(0,T;  L^2(\Omega))$ such that
we may estimate with H\"older's, Young's, and {Poincar\'e}'s inequality that
\begin{align*}
\langle \f f , \f v^n \rangle \leq \frac{\nu }{2}\|\nabla \f v^n \|_{L^2(\Omega)}^2 + \frac{C}{2\nu }\| \f f _1 \|_{ H^{-1}(\Omega)} ^2 + \| \f f_2 \|_{L^2(\Omega)}  \left  (\| \f v^n \|_{L^2(\Omega)}^2+1\right ) \,.
\end{align*}
Inserting this into~\eqref{energydis} allows to apply a version of Gronwall's Lemma in order to infer that $ \{ \f v^n \}$ is bounded and thus weakly$^*$ compact in $\mathbb{X}$ such that there exists a $ \f v \in \mathbb{X}$ with
\begin{align}
\f v^n \stackrel{*}{\rightharpoonup} \f v\quad  \text{in } \mathbb{X}\,.\label{weakkonv}
\end{align}
From~\eqref{energydis}, we observe 
\begin{align*}
\int_0^T \left | \frac{\de }{\de t }\|\f v^n \|^2_{L^2(\Omega)} \right |\de t  \leq 2\int_0^T \nu \| \nabla \f v^n \|_{L^2(\Omega)}^2 + | \langle \f f , \f v ^n \rangle | \de t 
\end{align*}
and by the boundedness of the sequence $\{ \f v^n\} $ in $\mathbb{X}$ as well as~\eqref{rhs}  that the sequence $\{ \|\f v^n \|^2_{L^2(\Omega)} \} _{n\in\N} $ is bounded in $\BV$. By Helly's selection principle, we may infer that there exists a function $E\in\BV$ such that
\begin{align}
\frac{1}{2} \|\f v^n(t) \|^2_{L^2(\Omega)} \ra E (t) \quad\text{for all~} t\in(0,T)\,.\label{BVconv}
\end{align}
\textit{Step 3, Discrete relative energy inequality: }
In order to show the convergence to energy-variational solutions, we derive a discrete version of the relative energy inequality. Assume $\vv \in \mathbb Y \cap \D$. 
Adding~\eqref{energydis} and~\eqref{vdis} tested with $- P_n \vv$,
we find
\begin{align}
\frac{1}{2}  \frac{\de }{\de t}\| \f v^n \|_{L^2(\Omega)}^2  + \nu\left  ( \nabla \f v^n ; \nabla \f v ^n - \nabla P_n\vv \right )   =  \langle \f f , \f  v^n-P_n \vv  \rangle + \left ( \t \f v^n , P_n \vv \right ) + \left ( ( \f v^n\cdot \nabla ) \f v^n , P_n \vv \right )  \,. \label{weakdis}
\end{align}
For the system operator~$\mathcal{A}_\nu$, we observe that
\begin{multline*}
\langle \mathcal{A}_\nu ( P_n \vv), \f v^n - P_n \vv\rangle =\\ \left ( \t P_n \vv , \f v^n \right ) - \t \frac{1}{2} \| P_n \vv \|_{L^2(\Omega)}^2 +\nu \left ( \nabla P_n \vv , \nabla \f v^n - \nabla P_n \vv \right ) + \left ( (P_n\vv \cdot \nabla ) P_n \vv , \f v^n \right ) - \langle \f f , \f v^n - P_n \vv \rangle\,.
\end{multline*}

Adding to as well as subtracting from~\eqref{weakdis} the term $\langle\mathcal{A}_\nu(P_n\vv) , \f v^n -  P_n \vv \rangle  $ leads to 
\begin{multline}
\frac{1}{2}  \frac{\de }{\de t} \| {\f v^n} - P_n\vv \|_{L^2(\Omega)}^2+ \left ( \nu \| \nabla \f v ^n - \nabla P_n \vv \|_{L^2(\Omega)}^2 + \langle  \mathcal{A}_\nu(P_n\vv) , \f v^n -P_n \vv \rangle\right )  \\
= \Big(  
\left ( ( \f v^n \cdot \nabla ) \f v^n , P_n \vv \right ) + \left ( ( P_n \vv \cdot \nabla) P_n \vv , \f v^n \right ) \Big )   \,. \label{disineqrel}
\end{multline}
By some algebraic transformations, we find
\begin{multline}
\left ( ( \f v^n \cdot \nabla ) \f v^n , P_n \vv \right ) + \left ( ( P_n \vv \cdot \nabla) P_n \vv , \f v^n \right ) \\
={}\left ( ( (\f v^n- P_n \vv) \cdot \nabla ) (\f v^n- P_n \vv)  , P_n \vv \right ) 
\\+\left ( ( P_n \vv \cdot \nabla ) (\f v^n- P_n \vv)  , P_n \vv \right ) + \left ( ( P_n \vv \cdot \nabla) P_n \vv , \f v^n - P_n \vv \right ) 
\,.
\label{discalc}
\end{multline}
For the first term on the right-hand side of~\eqref{discalc}, we observe
\begin{align*}
\nu \| \nabla \f v ^n - \nabla P_n \vv \|_{L^2(\Omega)}^2  - \left ( ( (\f v^n- P_n \vv) \cdot \nabla ) (\f v^n- P_n \vv)  , P_n \vv \right )  = \mathcal{W}_\nu(\f v^n | P_n \vv) - \mathcal{K}(P_n \vv)\mathcal{R}(\f v^n | P_n \vv) \,.
\end{align*}

For the second term on the right-hand side of~\eqref{discalc}, we find with an integration-by-parts (or the usual skew-symmetry in the second two variables of the trilinear convection term) that
\begin{align*}
\left ( ( P_n \vv \cdot \nabla ) (\f v^n- P_n \vv)  , P_n \vv \right ) + \left ( ( P_n \vv \cdot \nabla) P_n \vv , \f v^n - P_n \vv \right )  =0\,.
\end{align*}

In order to find the discrete version of the relative energy inequality, the term $\mathcal{K}_\nu(P_n\vv)\mathcal{R}(\f v | P_n\vv) $ is added and subtracted to~\eqref{disineqrel} and the resulting equality is integrated over $(s,t)$ such that
\begin{multline}
\mathcal{R}(\f v^n(t) | P_n \vv (t) ) + \int_s^t \left [ \mathcal{W}_\nu(\f v^n | P_n \vv ) + \langle  \mathcal{A}_\nu(P_n \vv ) , \f v^n -P_n \vv \rangle - \mathcal{K}(P_n\vv) \mathcal{R}(\f v^n| P_n\vv) \right ] \de s 
\\ \leq \mathcal{R}(\f v^n(s) | P_n \vv (s) ) \label{stability}
\end{multline}
 for a.e.~$s,\,t\in(0,T)$ and $\nu>0$.  

\textit{Step 4, Passage to the limit:}
 Via Lemma~\ref{lem:invar}, the inequality~\eqref{stability} may be written as
\begin{multline*}
-\int_0^T \phi' \mathcal{R}(\f v^n | P_n \vv )\de s  \\ + \int_0^T \phi \left [ \mathcal{W}_\nu(\f v^n | P_n \vv ) + \langle  \mathcal{A}_\nu(P_n  \vv ) , \f v^n -P_n \vv \rangle - \mathcal{K}(P_n\vv) \mathcal{R}(\f v^n| P_n\vv)\right ] \de s 
 \leq 0
\end{multline*}
for all $\phi\in{C}^1_c(0,T)$ with $\phi\geq 0$ a.e.~on $(0,T)$. 
Since $\C^1 ([0,T];\mathcal{C}_{c,\sigma}^\infty(\Omega;\R^d))$ is also dense in $\mathbb{Y}$, we may observe 
the strong convergence of the projection $P_n$, \textit{i.e}.,
\begin{align}
\| P_n \vv - \vv \|_{L^2(0,T;\V)} + \| P_n \vv - \vv \|_{L^2(0,T;L^{d/2}(\Omega))}\ra 0 \quad \text{as }n \ra \infty \quad \text{for all }\vv \in \mathbb Y \cap \D \,.\label{strongconv}
\end{align} 
This together with~\eqref{weakkonv} allows to pass to the limit in the second term via the weakly-lower semi-continuity of the convex functional  $\mathcal{W_\nu}$ (see~Lemma~\ref{lem:lowcon} and Remark~\ref{rem:W}).
Since $\f v^n$ only occurs linearly in the  term including $\mathcal{A}_\nu$ on the left-hand side, we may also pass to the limit in this term. 
Indeed, the time derivative may be interchanged with the projection $P_n$ such that
\begin{align*}
\left ( \t P_n \vv , \f v^n - P_n \vv \right ) = \left ( P_n \t \vv , \f v^n - P_n \vv \right )=\left (  \t \vv , \f v^n - P_n \vv \right )\,,
\end{align*}
where it was used that $P_n$ is an orthogonal projection. This together with~\eqref{strongconv} imply that the consistency error vanishes, \textit{i.e.},
\begin{align*}
\int_0^T \phi&\left  \langle \mathcal{A}_\nu( \vv) - \mathcal{A}_\nu (P_n\vv) , \f v^n - P_n\vv\right  \rangle e^{-\int_0^s \mathcal{K}_\nu (P_n \vv)\de \tau  }\de s  \\={}&
\nu\int_0^T \phi  \left ( \nabla \vv- \nabla P_n \vv ;\nabla  \f v^n -\nabla P_n \vv ) \right )   e^{-\int_0^s \mathcal{K}_\nu (P_n \vv)\de \tau } \de s  \\&+ \int_0^T\phi  \left ( ( (\vv- P_n \vv) \cdot \nabla) \vv  +(P_n \vv \cdot \nabla) (\vv - P_n \vv),\f v^n - P_n \vv \right )  e^{-\int_0^s \mathcal{K}_\nu (P_n \vv)\de \tau } \de s 
\\
\leq{}&\nu \| \phi \|_{L^\infty(\Omega)}\| \nabla \vv - \nabla P_n \vv \|_{L^2(0,T;L^2(\Omega))} \| \nabla \f v^n - \nabla P_n \vv \|_{L^2(0,T;L^2(\Omega))} \\
&+ \| \phi \|_{L^\infty(\Omega)}\| \vv -P_n \vv \|_{L^2(0,T;L^{d/2}(\Omega))} \| \nabla \vv \|_{L^\infty(0,T;L^{2d/(d-2)}(\Omega))}  \| \f v^n - P_n\vv \|_{L^2(0,T;L^{2d/(d-2)}(\Omega))} \\
&+  \| \phi \|_{L^\infty(\Omega)} \| P_n \vv\|_{L^\infty(0,T;L^d(\Omega))} \|\nabla \vv-\nabla P_n\vv\|_{L^2(0,T;L^2(\Omega))}  \| \f v^n - P_n\vv \|_{L^2(0,T;L^{2d/(d-2)}(\Omega))} 
\,.
\end{align*}
 Weak convergence of $\f v^n$ in $L^2(0,T;\V)$ implies that the norms of $\f v^n$ on the right-hand side are bounded independent of $n$. Note that dimension $d=2$ is excluded at this point. But the proof can also be adapted to dimension two.   The strong convergence~\eqref{strongconv} allows to pass to the limit on the right-hand side, which vanishes. 
The strong convergence of the projection $P_n$ to the identity on $\Ha$ as $n\ra\infty$ allows to pass to the limit in the initial values, too. 
Finally, we observe from~\eqref{weakkonv},~\eqref{BVconv}, and~\eqref{strongconv} that
\begin{align*}
\mathcal{R}(\f v^n | P_n \vv) \ra \mathcal{R}(\f v | \vv) + \xi \quad \text{a.e.~in }(0,T)\,,
\end{align*}
where $\xi := E -1/2\| \f v \|_{L^2(\Omega)}^2 $ for a.e.~$t\in (0,T)$. 
As a consequence, we observe that the relative energy inequality~\eqref{relen} holds in the limit a.e.~in $(0,T)$. 
Due to the additional regularity of Corollary~\ref{cor:addreg}, \textit{i.e.,} $E\in\BV$ and $\f v \in \C_w([0,T];\Ha)$, the relation even holds everywhere in $[0,T]$. 
%

\textit{Step 5, Vanishing viscosity limit $\nu\ra0$:}
Now, we focus on the case $\nu=0$. 
Therefore, we consider a sequence $\{ (\f v ^{\nu}, \xi^\nu) \}_{\nu \in (0,1)}$ of energy-variational solutions to the Navier-Stokes equations according to Theorem~\ref{thm:exist} for $\nu\ra 0$. 
These solutions fulfill Definition~\ref{def:diss} with $\mathcal{W}_\nu$ given by~\eqref{Wnu}. 
Inserting $ \vv = 0$ in this definition, we find the usual energy estimate~\eqref{monotone} such that with the usual estimates of the right-hand side, \textit{i.e.}, \eqref{rhs} with $\f f_1 =0$ (Note that $\mathbb{Z}_0 = L^\infty(0,T;\Ha)$), we deduce the weak convergence of a subsequence in the energy space, \textit{i.e.,} 
\begin{align*}
\f v^\nu \stackrel{*}{\rightharpoonup} \f v\quad  \text{in } \mathbb{X}_0
\end{align*}
 with $\mathbb{X}_0 $ as given above by $\mathbb{X}_0:=L^\infty(0,T;L^2_{\sigma} (\Omega)) $. Due to the additional regularity deduced in Corollary~\ref{cor:addreg}, we even deduce that 
 \begin{align*}
 \f v ^\nu \ra \f v \quad  \text{in } \C_w([0,T];\Ha)
 \end{align*}
 and thus pointwise for all $t\in(0,T)$. 
Similar as in the proof of Corollary~\ref{cor:addreg}, we infer that $E^\nu = \frac{1}{2}\| \f v^\nu \|_{L^2(\Omega)}^2 + \xi^\nu$ is a bounded sequence in $ \BV$ such that we may select via Helly's theorem an pointwise converging subsequence 
\begin{align*}
E^\nu  \ra E \quad \text{  pointwise everywhere in }[0,T]\,. 
\end{align*}
Note that the bound derived in Corollary~\ref{cor:addreg} does not depend on $\nu$ since the essential \textit{a priori} estimates are independent of $\nu$.

 With the usual skew-sym\-metry in the last two entries of the trilinear form, we find
\begin{align*}
-\left  ( (( \f v^\nu-\vv) \cdot \nabla ) (\f v^\nu-\vv)  , \vv \right )  =  \left ( ( \f v ^\nu - \vv ) \otimes ( \f v^\nu - \vv ) , ( \nabla \vv)_{\sym} \right ) \,.
\end{align*}
This can be used to rewrite the relative energy inequality~\eqref{relen}  into
\begin{multline}
\left [ E^\nu -\left ( \f v^\nu , \vv\right ) + \frac{1}{2}\| \vv\|_{L^2(\Omega)}^2\right ]\Big|_{s}^{t}   \\+ \int_s^t \left [ \nu \| \nabla \f v ^\nu - \nabla \vv \|_{L^2(\Omega)}^2 + \nu \left ( \nabla \vv ; \nabla \f v^\nu -\nabla \vv\right )   \right ] \de \tau\\+ \int_s^t \left [  \mathcal{W}_0(\f v^\nu| \vv) + \langle \mathcal{A}_0 (\vv) , \f v ^\nu - \vv\rangle  - \mathcal{K } (\vv) \left ( E^\nu -\left ( \f v^\nu , \vv\right ) + \frac{1}{2}\| \vv\|_{L^2(\Omega)}^2\right )  \right ] \de\tau \leq 0 \,\label{singularlim}
\end{multline}
for all  $\vv \in \mathbb{Y}_0\cap \D \cap L^2(0,T;H^2(\Omega)) $ and all $\nu>0$.
First, we decrease the number of admissible regularity measures $\mathcal{K}$, in order to make them independent of $\nu$, $\mathcal{K}_0: \mathbb{Y}_0 \ra \R_+$ is chosen such that $\mathcal{W}_0$ given by~\eqref{Wnull} is convex and weakly lower semi-continuous in $\f v$ and continuous in $\vv$ (see Definition~\ref{def:diss}). 
In the first and last line of~\eqref{singularlim}, we may pass to the limit by the weak$^*$ convergence of $\{ (\f v^\nu,E^\nu)\} $ in $ \mathbb{X}_0\cap \C_w([0,T];\Ha) \times \BV $, the lower semi-continuity of $\mathcal{W}_0$ and the linear occurrence in all other terms. 
For the second line, we observe the estimates
\begin{multline*}
\int_s^t\left ( \nu \| \nabla \f v ^\nu - \nabla \vv \|_{L^2(\Omega)}^2 + \nu \left ( \nabla \vv ; \nabla \f v^\nu -\nabla \vv\right )   \right ) \de \tau 
\\
\geq 
\int_s^t \left ( \nu \| \nabla \f v ^\nu - \nabla \vv \|_{L^2(\Omega)}^2 + \sqrt{\nu} \| \nabla \f v^\nu - \nabla \vv \|_{L^2(\Omega)} \sqrt \nu \| \nabla \vv \|_{L^2(\Omega)} \right ) \de t \\ \geq{} \int_s^t \left (  \frac{\nu}{2} \|  \nabla \f v ^\nu - \nabla \vv \|_{L^2(\Omega)}^2 \de s - \frac \nu 2 \| \nabla \vv \|_{L^2(\Omega)} \right) \de \tau 
\\ 
\geq{} -\frac \nu 2 \| \nabla \vv \|_{L^2(\Omega\times (s,t))} 
 \ra 0 \qquad \text{as }\nu \ra 0\,.
\end{multline*}
We infer that the relative energy inequality~\eqref{relen} is fulfilled  in the limit $ \nu \ra 0 $. This proves the existence of energy-variational solutions to the Euler equations and thus the assertion. 
In order to allow more general test functions, \textit{i.e.}, $ \vv \in\mathbb{Y}_0 \cap \Do $ instead of $\mathbb{Y}\cap \D $ one may use usual density arguments and the continuity of $\mathcal{K}_0$ in $\mathbb{Y}_0 \cap \Do$.

\textit{Step 6, Solution set:} Let $(\f v, E)$ be an energy-variational solution according to Defintion~\ref{def:diss}.
First, we observe that $E$ is indeed a function of bounded variation according to Corollary~\ref{cor:addreg}. 
Secondly, let $(\f v^1 ,E^1)$ and $(\f v^2 ,E^2 )$ be in $\mathcal{S}(\f v _0 , \f f)$. Then their convex combination 
$(\f v^\lambda ,E^\lambda)= (\lambda \f v^1 + ( 1-\lambda )\f v^2 , \lambda E^1 + (1-\lambda)E^2) $ for $\lambda \in (0,1)$ is again associated to an energy-variational solution. 
Indeed, we find by the linearity in $E$ and $\f v$  and the convexity of $\mathcal{W}_\nu$ in $\f v$ that
\begin{align*}
\Big [E ^\lambda& -{} ( \f v^\lambda , \vv) +\frac{1}{2}\| \vv\|_{L^2(\Omega)}^2\Big ]\Big|_s^t \\+& \int_s^t \left [ \mathcal{W}_\nu(\f v^\lambda , \vv )   + \left \langle \mathcal{A}_\nu(\vv ) , \f v^\lambda - \vv  \right \rangle - \mathcal{K}(\vv)\left  (E^\lambda - ( \f v^\lambda , \vv) +\frac{1}{2}\| \vv\|_{L^2(\Omega)}^2\right  ) \right ]  \de \tau 
\\
= {}&
 \lambda \Big [E ^1 - ( \f v^1 , \vv) +\frac{1}{2}\| \vv\|_{L^2(\Omega)}^2\Big ]\Big|_s^t\\ &+\lambda \int_s^t \left [ \mathcal{W}_\nu(\f v^\lambda , \vv )   + \left \langle \mathcal{A}_\nu(\vv ) , \f v^1 - \vv  \right \rangle - \mathcal{K}(\vv)\left  (E^1 - ( \f v^1 , \vv) +\frac{1}{2}\| \vv\|_{L^2(\Omega)}^2\right  ) \right ]  \de \tau \\
&+ (1-\lambda ) \Big [E ^2- ( \f v^2, \vv) +\frac{1}{2}\| \vv\|_{L^2(\Omega)}^2\Big ]\Big|_s^t \\&+(1-\lambda) \int_s^t \left [ \mathcal{W}_\nu(\f v^\lambda , \vv )   + \left \langle \mathcal{A}_\nu(\vv ) , \f v^2 - \vv  \right \rangle - \mathcal{K}(\vv)\left  (E^2 - ( \f v^2 , \vv) +\frac{1}{2}\| \vv\|_{L^2(\Omega)}^2\right  ) \right ]  \de \tau 
\\
\leq {}&
\lambda \Big [E ^1 - ( \f v^1 , \vv) +\frac{1}{2}\| \vv\|_{L^2(\Omega)}^2\Big ]\Big|_s^t
 \\&+ 
 \lambda\int_s^t \left [ \mathcal{W}_\nu(\f v^1 , \vv )   + \left \langle \mathcal{A}_\nu(\vv ) , \f v^1 - \vv  \right \rangle - \mathcal{K}(\vv)\left  (E^1 - ( \f v^1 , \vv) +\frac{1}{2}\| \vv\|_{L^2(\Omega)}^2\right  ) \right]  \de \tau \\
&+ (1-\lambda ) \Big [E ^2- ( \f v^2, \vv) +\frac{1}{2}\| \vv\|_{L^2(\Omega)}^2\Big ]\Big|_s^t 
\\&+
(1-\lambda)  \int_s^t \left [ \mathcal{W}_\nu(\f v^2 , \vv )   + \left \langle \mathcal{A}_\nu(\vv ) , \f v^1- \vv  \right \rangle - \mathcal{K}(\vv)\left  (E^2 - ( \f v^2 , \vv) +\frac{1}{2}\| \vv\|_{L^2(\Omega)}^2\right  ) \right ]  \de \tau \leq 0 \,.
\end{align*}

Finally, we want to show that any sequence $\{ \f v^n , E^n \} _{n\in\N} \subset \mathcal{S}(\f v _0 , \f f) $ admits a cluster point in the solution set $\mathcal{S}(\f v _0 , \f f)$. For any sequence  $\{ \f v^n , E^n \} _{n\in\N} \subset \mathcal{S}(\f v _0 , \f f) $, we infer from~\eqref{relen} by choosing $\vv=0$ and $s=0$ the standard \textit{a priori} estimates
\begin{align*}
E^n (t) + \int_0^t \nu \| \nabla \f v^n \| _{L^2(\Omega)}^2 \de s \leq \frac{1}{2} \| \f v_0 \|_{L^2(\Omega)}^2+ \xi (0)  + \int_0^t \langle \f f , \f v^n  \rangle \de s \,.
\end{align*}
From the estimate~\eqref{rhs}, we infer the \textit{a priori} bounds $\{ ( E^n , \f v^n)\} $ in $L^\infty(0,T)\times L^2(0,T;\V)$ in the case $\nu>0$ and $\{ E^n\}$ in $L^\infty(0,T)$ for $\nu=0$. From the definition of $E$, we infer that $E^n \geq 1/2\| \f v^n \|_{L^2(\Omega)}^2 $  on $[0,T]$. This implies that $\{ \f v_n \} $ is bounded in $\mathbb{X}$. Moreover from the Corollary~\ref{cor:addreg}, we infer the boundedness of the sequence $\{ E^n\}$ in $\BV$. Helly's selection principle allows to select an  everywhere in $[0,T]$ converging  subsequence to some limit $E\in\BV$. Furthermore, we infer $\f v^n \stackrel{*} {\rightharpoonup} \f v $ in $\mathbb{X}$. Due to the additional regularity from Corollary~\ref{cor:addreg}, we even infer $\f v^n \ra \f v $ in $\C_w([0,T];\Ha)$. 
Reformulating the relative energy inequality via Lemma~\ref{lem:invar}, allows to pass to the limit in the resulting formulation. 
Indeed, the strong convergence of $E^n$ and the weak convergence of $\f v^n$ imply 
\begin{multline*}
\lim_{n\ra\infty} \int_s^t  \left ( E^n - ( \f v ^n , \vv) + \frac{1}{2}\| \vv \|_{L^2(\Omega)}^2 \right ) \de \tau 
\\+ \lim_{n\ra \infty} \int_s^t  \left ( \left \langle \mathcal{A}(\vv), \f v^n - \vv \right \rangle - \mathcal{K}(\vv) \left ( E^n - ( \f v ^n , \vv) + \frac{1}{2}\| \vv \|_{L^2(\Omega)}^2 \right ) \right ) \de \tau 
\\
= \int_s^t  \left ( E - ( \f v  , \vv) + \frac{1}{2}\| \vv \|_{L^2(\Omega)}^2 \right ) \de \tau  
\\+  \int_s^t  \left ( \left \langle \mathcal{A}(\vv), \f v - \vv \right \rangle - \mathcal{K}(\vv) \left ( E - ( \f v  , \vv) + \frac{1}{2}\| \vv \|_{L^2(\Omega)}^2 \right ) \right ) \de \tau \,
\end{multline*}
for all $s<t \in (0,T)$. 
The weak convergence of $\{ \f v^n\}$ and the weakly lower semi continuity of $\mathcal{W}_\nu$  imply 
\begin{align*}
\liminf_{n\ra\infty} \int_s^t  \mathcal{W}_\nu ( \f v^n  | \vv ) \de \tau \geq \int_s^t  \mathcal{W}_\nu ( \f v | \vv ) \de \tau 
\end{align*}
for all $\vv \in \mathbb{Y}\cap \D$ and for all $s<t \in (0,T)$. 
 This implies that the limit $(\f v,E)$ is again an energy-variational solution according to Definition~\ref{def:diss} via the reformulation $\xi = E-1/2\| \f v\|_{L^2(\Omega)}^2$. Again the condition $\xi \geq 0$ is fulfilled, due to 
\begin{align*}
E(t)= \lim _{n\ra \infty}  E^n(t)  \geq \liminf _{n\ra \infty} \frac{1}{2}\| \f v^n (t) \|_{L^2(\Omega)}^2 \geq  \frac{1}{2}\| \f v (t) \|_{L^2(\Omega)}^2 
\end{align*}
for all $t\in[0,T]$.

\textit{Step 7, Continuous dependence:}
In order to prove the claimed continuity of the set-valued mapping, we need to prove two assertions.
Therefore, let
\begin{align}
\f f^n \ra \f f \text{ in }\mathbb{Z}\quad \text{and}\quad \f v ^n_0 \ra \f v_0 \text{ in }\Ha \,.\label{convdata}
\end{align}
  We equip the domain $  \Ha \times \mathbb{Z}$ with the strong topology and the range $\mathbb{X}
  \times \BV $ with the weak$^*$ topology. 
  
 Introducing the Kuratowski limits 
\begin{align*}
\Ls_{(\f v_0^n, \f f^n)  \ra (\f v_0 ,\f f) } \mathcal{S}(\f v_0^n, \f f^n ) :={}& \Big \{ (\f v, E) \in \mathbb{X}
\times \BV ; \\
&\quad\text{ there exsits a sequence } ( \f v_0^n, \f f^n) \ra (\f v_0,\f f) \text{ and }(\f v^n, E^n) \stackrel{*}{\rightharpoonup}(\f v, E) \\& \quad \text{ such that } (\f v^n, E^n)\in \mathcal{S}(\f v_0^n , \f f^n )\Big \}\,,
\\
\Li_{(\f v_0^n, \f f^n)  \ra (\f v_0 ,\f f) } \mathcal{S}(\f v_0^n, \f f^n ) :={}& \Big \{ (\f v, E) \in \mathbb{X}
\times \BV ;\\&\quad  \text{ for all } (\f v_0^n, \f f^n)  \ra (\f v_0 ,\f f) \text{ there exsists }(\f v^n, E^n) \in \mathcal{S}( \f v_0^n, \f f^n)\\& \quad \text{ such that } (\f v^n, E^n) \stackrel{*}{\rightharpoonup}(\f v, E) \Big \}
\end{align*}
we need to prove that $$\Ls_{(\f v_0^n, \f f^n)  \ra (\f v_0 ,\f f) } \mathcal{S}(\f v_0^n, \f f^n ) \subset \mathcal{S}(\f v_0 , \f f )\text{ and }\Li_{(\f v_0^n, \f f^n)  \ra (\f v_0 ,\f f) } \mathcal{S}(\f v_0^n, \f f^n )\supset \mathcal{S}(\f v _0,\f f) \,,$$ which are called upper and lower semi-continuity, respectively. 
A set-valued map is said to be continuous, if $\Ls_{(\f v_0^n, \f f^n)  \ra (\f v_0 ,\f f) } \mathcal{S}(\f v_0^n, \f f^n )=\Li_{(\f v_0^n, \f f^n)  \ra (\f v_0 ,\f f) } \mathcal{S}(\f v_0^n, \f f^n )$. 

\textit{Step 7.1, Upper semi-continuity:} 
First, we show that $\Ls_{(\f v_0^n, \f f^n)  \ra (\f v_0 ,\f f) } \mathcal{S}(\f v_0^n, \f f^n )\subset \mathcal{S}(\f v _0,\f f)$. 
 From the steps 1 to 5, we infer that to every $n\in\N$ there exists an energy-variational solution $(\f v^n , E^n)$ according to Definition~\ref{def:diss}.
Since $ (\f v_0^n , \f f^n) $ is bounded in $\Ha \times \mathbb{Z}$, we infer the boundedness of 
the sequence  $\{( \f v^n,E^n )\}_{n\in\N} $ in $ \mathbb{X}\times \BV 
$ in the same way as in the Corollary~\ref{cor:addreg}. 
Thus, there exists some $(\f v, E )$ and some subsequence $\{ ( \f v^{n_k}, E^{n_k})\}$ such that  $ ( \f v^{n_k}, E^{n_k})\stackrel{*}{\rightharpoonup} (\f v ,E) $ in $\mathbb{X}\times \BV
$. 
For every $\psi \in \C^1_c((0,T)) $ with $\psi \geq 0$, we observe that
\begin{align*}
\liminf_{n\ra \infty} \int_0^T \psi \left ( \frac{1}{2}\| \f v ^{n_k} \| _{L^2(\Omega)}^ 2- E^{n_k} \right ) \de t \geq \int_0^T \psi \left ( \frac{1}{2}\| \f v  \| _{L^2(\Omega)}^ 2- E \right ) \de t 
\end{align*}
from the strong convergence of $\{E^{n_k}\}$ in $L^1(0,T)$ and the weakly lower semi-continuity of the $L^2$-norm. 
Similar, we observe for every fixed $\phi \in   \C^1_c((0,T)) $ with $\phi \geq 0$, and $\vv \in \mathbb{Y}$ that 
\begin{multline*}
-\limsup_{n\ra \infty}  \int_0^T \phi' \left [E^{n_k} - ( \f v ^{n_k}, \vv ) \right ] \de t + \liminf_{n\ra \infty}\Bigg [  \int_0^T \phi  \mathcal{K}(\vv ) \left [ \frac{1}{2}\| \f v ^{n_k}  \| _{L^2(\Omega)}^2 - E^{n_k}  \right ]\de t \\  \int_0^T \phi \left [ \nu \left ( \nabla \f v^{n_k} , \nabla \f v^{n_k}-\nabla \vv\right ) + \left (\f v^{n_k} \otimes \f v^{n_k} ; \nabla \vv\right ) +\left  ( \t \vv , \f v^{n_k} \right ) - \left \langle \f f ^{n_k} , \f v^{n_k} - \vv \right \rangle \right ] \de t \Bigg]
\\
\geq - \int_0^T \phi' \left [E - ( \f v , \vv ) \right ] \de t + \int_0^T \phi  \mathcal{K}(\vv ) \left [ \frac{1}{2}\| \f v   \| _{L^2(\Omega)}^2 - E  \right ]\de t 
\\
+ \int_0^T \phi \left [ \nu \left ( \nabla \f v , \nabla \f v-\nabla \vv\right ) + \left (\f v \otimes \f v ; \nabla \vv\right ) +\left  ( \t \vv , \f v \right ) - \left \langle \f f  , \f v - \vv \right \rangle \right ] \de t \,.
\end{multline*}
Note that the only difference in this limit in comparison to the proof of \textit{Step 6} is that $\f f^{n_k}$ converges strongly now, which together with the weak convergence of $\{\f v^{n_k}\} $ implies the convergence of their product. 
This implies that the limit $(\f v,E )$ is an energy-variational solution again and therewith $(\f v,E ) \in \mathcal{S}(\f v_0,\f f)$.

\textit{Step 7.1, Lower semi-continuity:} 
Secondly, we show that $\Li_{(\f v_0^n, \f f^n)  \ra (\f v_0 ,\f f) } \mathcal{S}(\f v_0^n, \f f^n )\supset \mathcal{S}(\f v _0,\f f)$.
Therefore, we have to construct a recovery sequence. For a given $( \f v_0^n , \f f^n ) \ra (\f v_0,\f f) $, we construct $\bar{\f v}^n \in L^2(0,T;\V)\cap H^1(0,T;(\V)^*)\cap\C([0,T];\Ha)$ as the solution to the Stokes problem 
\begin{align}\label{recov}
\t \bar{\f v}^n - \nu \Delta \bar{\f v}^n + \nabla \bar{p}^n ={}& \f f^n - \f f \quad&& \text{in } \Omega \times (0,T)
\\
\di \bar{\f v}^n ={}& 0\quad &&\text{in }\Omega \times (0,T)\notag
\\
\bar{\f v}^n ={}& 0 \quad &&\text{on }\partial \Omega \times (0,T)\notag
\\
\bar{\f v}^n (0) ={}& \f v^n_0 - \f v _0 \quad&&\text{in }\Omega \,.\notag
\end{align}
Note that this linear PDE problem can be solved by Lions theorem  (see\cite[Thm.~11.3]{renardy}) on linear parabolic problems in the usual weak sense in the space indicated above. Since Lions theorem also guarantees the continuous dependence, we infer that
\begin{align}
\bar{\f v}^n  \ra  0 \quad \text{in }   \mathbb{X}\cap\C([0,T];\Ha)\,.\label{convvbar}
\end{align}
Note that in the case $\nu=0$, the problem~\eqref{recov} is solved by the choice 
$ \bar{\f v}^n (t) = \f v^n_0-\f v_0 + \int_0^t P(\f f^n - \f f) \de s $, where $P$ denotes the Leray-projection onto solenoidal functions in $L^2(\Omega)$. 
Now, we may choose $\f v^n = \f v+ \bar{\f v}^n$. By construction it holds that $$\f v^n(0) = \f v^n _0\text{ and } \f v^ n \ra \f v\text{ in } \mathbb{X}\cap\C_w([0,T];\Ha)\,.$$
In order to infer a condition on $E^n$, we consider the relative energy inequality for $\f v^n$,
\begin{multline}
\left [E^n - ( \f v^n , \vv) \right ] \Big|_s^t+\int_s^t\mathcal{K}(\vv) \left [\frac{1}{2}\| \f v^n \|^2_{L^2(\Omega)}- E^n \right ]\de \tau \\
+\int_s^t \left [  \nu \left (\nabla \f v ^n , \nabla \f v^n - \nabla \vv \right ) + \left ( \f v^n \otimes \f v^n ; \nabla \vv \right ) + \left ( \t \vv, \f v^n \right ) - \left \langle \f f ^n , \f v ^n - \vv \right \rangle  \right ]\de \tau \\
=\left [E - ( \f v , \vv)  \right ] \Big|_s^t +\int_s^t \mathcal{K}(\vv ) \left [ \frac{1}{2}\| \f v  \| _{L^2(\Omega)}^2 - E  \right ] \de \tau
\\+ \int_s^t \nu \left ( \nabla \f v , \nabla \f v - \nabla \vv \right ) + \left ( \f v (t) \otimes \f v(t) ; \nabla \vv (t) \right ) + \left ( \t \vv , \f v \right ) - \langle \f f , \f v - \vv \rangle \de \tau 
\\ + \left [ \frac{1}{2}\| \bar{\f v}^n \|_{L^2(\Omega)}^2 - \left ( \bar{\f v}^n , \vv\right ) \right ] \Big|_s^t + \int_s^t \nu \left ( \nabla \bar{\f v}^n , \nabla \bar{\f v}^n - \nabla \vv\right ) + \left ( \t \vv, \bar{\f v}^n \right ) - \left \langle \f f^n - \f f, \bar{\f v}^n - \vv \right \rangle \de \tau 
\\
+ \left [ \bar{E}^n - \frac{1}{2}\| \bar{\f v}^n \|_{L^2(\Omega)}^2 \right ] \Big|_s^t 
+\int_s^t \mathcal{K}(\vv) \left [\left ( \f v , \bar{\f v}^n \right ) + \frac{1}{2}\| \bar{\f v}^n\|_{L^2(\Omega)}^2 - \bar{E}^n\right ]\de \tau 
\\+ \int_s^t 2\nu \left ( \nabla \f v , \nabla \bar{\f v}^n\right ) +  \left ( 2\bar{\f v}^n \otimes \f v + \bar{\f v}^n \otimes\bar{\f v}^n ; \nabla \vv\right )- \left \langle \f f ^n - \f f , \f v \right \rangle - \left \langle \f f, \bar{\f v}^n \right \rangle \de \tau \,.
\end{multline}
Since $(\f v, E) \in \mathcal{S}(\f v_0, \f f)$, the first two lines on the right-hand side of the previous inequality are non-positive, and for $\bar{\f v}^n $ as a weak solution to~\eqref{recov}, the third line on the right-hand side is zero. Thus it remains to choose $ \bar{E} ^n$ in such a way that the right-hand side is non-positive. Therefore, we consider the estimate 
$$  \left ( 2\bar{\f v}^n \otimes \f v + \bar{\f v}^n \otimes\bar{\f v}^n ; \nabla \vv\right ) \geq - \mathcal{K}(\vv) \left ( \| \f v \|_{L^2(\Omega)}\| \bar{\f v}^n \|_{L^2(\Omega)} + \frac{1}{2}\| \bar{\f v}^n \|_{L^2(\Omega)}^2 \right )\,,$$
in order to infer the conditions 
\begin{subequations}\label{conditons}
\begin{align}
\bar{E}^n (t) \geq {}& 2 \| \bar{\f v}^n(t) \|_{L^2(\Omega)}\| \f v (t)\|_{L^2(\Omega)} \quad\text{for all }t \in [0,T] 
\intertext{and}
\left [\bar{E}^n- \frac{1}{2}\| \bar{\f v}^n \|_{L^2(\Omega)}^2 \right ] \Big|_s^t &+ \int_s^t 2 \nu \left (\nabla \f v,\nabla \bar{\f v}^n \right ) - \langle \f f^v-\f f , \f v \rangle - \langle \f f , \bar{\f v}^n \rangle \de \tau  \leq 0 \quad\text{for all }s<t \in[0,T]\,.
\end{align}
\end{subequations}
One possible choice to define $\bar{E}^n$ would be to set 
\begin{align*}
\bar{E}^n(0) := \max_{t\in [0,T]} \Big [&2 \| \bar{\f v}^n \|_{\C([0,T];L^2(\Omega))}\| \f v \|_{L^\infty(0,T;L^2(\Omega))}  - \frac{1}{2}\| \bar{\f v}^n(t) \|_{L^2(\Omega)}^2 \\&+ \int_0^t 2 \nu \left (\nabla \f v,\nabla \bar{\f v}^n \right ) - \langle \f f^v-\f f , \f v \rangle - \langle \f f , \bar{\f v}^n \rangle \de s + \frac{1}{2}\| \bar{\f v}^n_0\|_{L^2(\Omega)}^2  \Big ]
\end{align*}
and define 
\begin{align*}
\bar{E}^n(t) = \bar{E}^n(0) - \frac{1}{2}\| \bar{\f v}^n_0\|_{L^2(\Omega)}^2 + \int_0^t 2 \nu \left (\nabla \f v,\nabla \bar{\f v}^n \right ) - \langle \f f^v-\f f , \f v \rangle - \langle \f f , \bar{\f v}^n \rangle \de s +  \frac{1}{2}\| \bar{\f v}^n(t) \|_{L^2(\Omega)}^2 \,.
\end{align*}
By this choice the conditions~\eqref{conditons} are fulfilled and additionally the convergences~\eqref{convvbar} and~\eqref{convdata} as well as the boundedness of $\f v$ in $\mathbb{X}$ allow to deduce that 
\begin{align*}
\bar{E}^n  \ra  0 \quad \text{ in $\BV$ and everywhere on }[0,T]\,,
\end{align*}
which implies the assertion.


 \end{proof}
 \begin{remark}
 The proof for the Euler equation is also possible via the Galerkin scheme. 
 
 The Galerkin proof can also be seen as a version of Lax theorem. If the scheme is stable with respect to the relative energy inequality, \textit{i.e.}, fulfills~\eqref{stability} and the consistency error vanishes, \textit{i.e.}, $\mathcal{A}_n(\vv) - \mathcal{A}_n(P_n \vv) \ra 0 $  in $\mathbb{X}_\nu^*$ as $n\ra \infty$ for smooth functions, then the numerical scheme converges. The discrete system operator $\mathcal{A}_n$ may be chosen differently. 
 
 The continuous dependence result is presented in a set-theoretic sense via convergence in the Kuratowski sense. These convergences are introduced in~\cite[Sec.~29]{Topology}. 
The connection of Kuratowski convergence  of the epigraphs of a sequence of convex functions to Gamma convergence of the associated convex functions is considered in~\cite[Thm.~4.16]{DalMaso}.
The continuity  of the set-valued map is consistent with the usual definition (see~\cite[Sec.~1.4]{Aubin}).

 \end{remark}
 \subsection{Well-posedness of minimal  energy-variational solutions\label{sec:well}}
 \begin{proof}[Proof of Theorem~\ref{thm:maxdiss}]
 For convenience, the proof is divided into several steps. 
 
 \textit{Step 1: Local argument.}
For the minimal energy-variational solution $(\f v , \xi)\in  \mathbb{X}\times L^\infty (0,T)$, we observe that for a~$t_0\in(0,T)$ with $\xi (t_0+) >0$, we may restart the process in $t_0$. 
If we call $(\bar{\f v}, \bar{\xi})$ the energy-variational solution on $(0,T-t_0 )$ starting from the initial value $\f v(t_0)$ then it holds
 \begin{multline*}
\mathcal{R}(\bar{\f v}( t) | \vv(t) ) +\bar \xi(t)+ \int_s^t \left ( \mathcal{W}_\nu(\bar{\f v} , \vv )   + \left \langle \mathcal{A}_\nu(\vv ) , \bar{\f v} - \vv  \right \rangle - \mathcal{K}(\vv) (\mathcal{R}(\bar{ \f v}| \vv ) +\bar\xi) \right )  \de \tau 
\\
\leq \mathcal{R}( \bar{\f v}( s) | \vv(s) )+ \bar \xi (s)  
\end{multline*}
for a.e.~$t$, $s\in (0,T-t_0)$. The concatenation $(\f v^2,\xi^2)$ defined by
\begin{align}\label{constructv2}
\begin{cases}
(\f v^2(t),\xi^2(t))= (\f v(t),\xi(t))\quad & \text{for }t\in(0,t_0)
\\
(\f v^2(t),\xi^2(t))= (\bar{\f v}(t-t_0),\bar \xi(t-t_0))\quad & \text{for }t\in(t_0,T)\,.
\end{cases}
\end{align}
 is then again an energy-variational solution on $(0,T)$, due to the fact that  $\f v^2 \in \C_w([0,T];\Ha)$ and
\begin{align*}
\lim_{t\searrow t_0 } \left ( \frac{1}{2}\| \f v^2(t) \|_{L^2(\Omega)}^2 + \xi^2(t) \right ) \leq \frac{1}{2}  \| \f v(t_0) \|_{L^2(\Omega)}^2 < \frac{1}{2}  \| \f v(t_0+) \|_{L^2(\Omega)}^2 + \xi(t_0+) \,.
\end{align*}
But this is violating the admissibility criterion of Definition~\ref{def:ad} for $t=t_0$. Thus it holds that $\xi(t_0) = 0 $ for every minimal energy-variational solution. 
Note that according to the proof of Theorem~\ref{thm:exist}, we may chose an energy-variational solution with $\xi(0)=0$.

Assume now, that there are two minimal energy-variational solutions $(\f v^1 , 0)$, $(\f v^2,0) \in \mathbb{X}\times L^\infty(0,T)$ such that $$ \frac{1}{2} \| \f v ^1(t_0) \|_{L^2(\Omega)} ^2 = E(t_0 +) = \frac{1}{2}\| \f v^2(t_0) \|_{L^2(\Omega)}^2 \quad \text{but}\quad \f v^1 \neq \f v^2 \,.$$
Due to the convexity of the solution set~(cf.~Theorem~\ref{thm:exist}), we observe that $\f v ^\lambda = \lambda\f v^1 + (1-\lambda)\f v^2$   is also an energy-variational solution according to Definition~\ref{def:diss} for all $\lambda \in [0,1]$ with the associated defect function vanishing $\xi^\lambda(t) =0 $. 
But due to the fact that for $\lambda \in (0,1)$, the energy decreases, \textit{i.e.},
\begin{align*}
E^\lambda (t_0+) = \frac{1}{2}\| \f v ^\lambda(t_0) \|_{L^2(\Omega)}^2 < \frac{ \lambda}{2}\| \f v^1(t_0) \|_{L^2(\Omega)}^2 + \frac{(1-\lambda)}{2} \| \f v^2(t_0) \|_{L^2(\Omega)}^2 = E(t_0+) \quad\,
\end{align*}
for $\lambda \in (0,1)$.
The functions $\f v^1$ and $\f v^2$ can not be admissible or it holds $\f v^1 = \f v^2 $. Thus the minimal energy-variational solution is unique in $t_0$. 

 \textit{Step 2: Restarting.}
From Theorem~\ref{thm:exist}, we infer that to every $\f v(t_0) \in \Ha$, there exists at least one energy-variational solution on $[0,T-t_0 ]$ with $\xi(0)=0$. Additionally, from Theorem~\ref{thm:exist} we infer that the associated  set  of solutions is  compact in the weak$^*$-topology. 
In the proof of Theorem~\ref{thm:exist}, we proved sequential compactness. Note that the weak$^*$-topology in $\BV$ is determined by the topology of Radon measures $\mathcal{M}([0,T])$, since the derivatives of the $\BV$-functions are Radon measures. This space is the dual of the continuous functions $\C([0,T])^* = \mathcal{M}([0,T])$, which are separable such that in the associated weak$^*$-topology sequential compactness coincides with (covering) compactness. 

Similarly, the spaces $L^2(0,T;\V)$ for $\nu >0$ and $L^2(0,T;\Ha)$ for $\nu = 0$
 are reflexive and thus for the associated weak topology sequential and covering compactness coincides due to the Theorem by Eberlein--\v{S}mulian (see~\cite{Eberlein}). 
Since the solution set is also bounded, it is by the Banach--Alaoglu theorem a closed subset of a compact set and as such compact itself. The procedure of restarting in a given point $t_0\in(0,T)$, hence provides a compact set of solutions. Now let $0<s<t<T$. Then it holds $ \mathcal{S}(\f v(t) ,\f f (\cdot +t ) ) \subset \mathcal{S}(\f v(s), \f f(\cdot  +s))$, for  $ (\f v (t), \xi (t))\in\mathcal{S}(\f v(s), \f f(\cdot  +s))  $. Indeed,  by the construction especially due to the property of the solution due to Remark~\ref{rem:semi} every energy-variational solution restarting in $t$ is also an energy-variational solution starting from $s$.

\textit{Step 3: Selection.}
Indeed, for every $t\in (0,T)$ there is a unique $\f v(t)$ such that starting from this initial value, we infer the solution set $\mathcal{S}(\f v(t) , \f f (\cdot + t) )$, which is compact in the weak$^*$ topology.
 The intersection of the solution sets 
$$ \mathcal{S}:=\bigcap_{t\in[0,T]}\mathcal{S}(\f v(t) , \f f (\cdot + t) )$$ is thus the intersection of nested compact  sets and as such nonempty and compact again.
From the point-wise uniqueness, also the global uniqueness follows. Thus, we constructed a unique solution which fulfills Definition~\ref{def:ad}. 
Moreover, by the construction it holds $E(t+) = \frac{1}{2}\| \f v (t) \|_{L^2(\Omega)}^2 $ and thus, $\xi(t) = 0$. 
From Proposition~\ref{prop:equation} it follows that the unique minimal energy-variational solution is actually a weak solution according to Definition~\ref{def:weak}.

 \textit{Step 4: Continuous dependence.}
 Now let again $(\f v_0^n , \f f^n)$ be such that the convergence~\eqref{convdata} holds. Then to every $n\in \N$ there exists a unique minimal energy-variational solution $\f v^n\in \mathbb{X}$ according to the previous steps of the proof. Now we want to consider the sequence $\{ \f v^n\}\subset \mathbb{X}$. Since every minimal energy-variational solution is also an energy-variational solution, \textit{i.e.}, $ (\f v^n , 0) \in \mathcal{S}(\f v_0^n , \f f^n)$, we infer by the continuous dependency result from Theorem~\ref{thm:exist} that the limit of  a sub sequence, say $(\bar{\f v}, \bar{E} )$ is again an energy-variational solution $(\bar{\f v}, \bar{E} ) \in \mathcal{S}(\f v_0, \f f)$.  Additionally, the convergence $ 1/2\| \f v^n (t) \|_{L^2(\Omega)}^2 \ra \bar{E}(t)$ holds  for all $t\in [0,T]$.
 Due to the definition of minimal energy-variational solutions of Definition~\ref{def:ad}, we infer that  for all $t\in [0,T]$, it holds that $ 1/2\|\f v (t)\|_{L^2(\Omega)}^2 \leq \bar{E}(t)$.

 Assume  now that there exists a $t_0 \in [0,T]$ such that $ 1/2 \| \f v(t_0) \|_{L^2(\Omega)}^2 < \bar{E}(t_0)$. Then, by the continuous dependence result of Theorem~\ref{thm:exist}, there exists a sequence 
$ \{ (\f u^n , F^n ) \} $ such that $  (\f u^n , F^n ) \in \mathcal{S}(\f v_0^n, \f f^n ) $ and $ F^n(t) \ra 1/2\| \f v (t) \|_{L^2(\Omega)}^2 $ for all $ t \in [0,T]$.
 Hence, there exists a $ N\in \N$ such that for all $n\geq N $ it holds that $ F^n(t_0) <1/2\| \f v^n (t_0) \|_{L^2(\Omega)}^2$. But this contradicts the fact that all $ \f v^n$ are minimal energy-variational solutions, since the solution $  (\f u^n , F^n ) \in \mathcal{S}(\f v_0^n, \f f^n ) $ is an admissible energy-variational solution. 
 
 We conclude that $1/2 \| \f v(t_0) \|_{L^2(\Omega)}^2 = \bar{E}(t_0) $ and due to the uniqueness $ \bar{\f v} = \f v$. 
 From  the pointwise convergence of the norms $ \| \f v^n(t) \|^2_{L^2(\Omega)} \ra \| \f v(t) \| _{L^2(\Omega)}^2 $ for all $t\in [0,T]$, one may concludes together with a suitable dominating function  that $ \| \f v^n \|_{L^2(0,T;L^2(\Omega))}^2 \ra \| \f v \|_{L^2(0,T;L^2(\Omega))}^2$.
 Together with the weak convergence $ \f v^n \rightharpoonup \f v $ in $L^2(0,T;L^2(\Omega))$ one may deduces the strong convergence 
 \begin{align*}
 \f v^n \ra \f v  \quad \text{in } L^2(0,T;L^2(\Omega))\,,
 \end{align*}
 by the uniform convexity of $L^2(0,T;L^2(\Omega))$ (see~\cite[Prop.~3.32]{brezisbook}).
\end{proof}

\begin{remark}[Comparison to measurable selections]
The selection argument of the previous proof is very similar to other selections principles (see for instance~\cite{maxbreit}). In the previous proof, we did not consider the measurablilty of the selected set in the Hausdorff metric of the associated power set.   
Nevertheless, the solution set $\mathcal{S}(\f v _0 ,\f f) $  is measurable with respect to the topology induced by the Hausdorff measure on the product topology, since the associated set-valued mapping is continuous (see~\cite[Prop.~8.21]{Aubin}).  Moreover the selected minimal energy-variational solution is measurable too, since it continuously depends on the data in the associated topologies. 
\end{remark}

\end{document}